\documentclass[12pt]{amsart}

\usepackage{amsmath,amssymb,amsthm,mathrsfs,enumerate,color,mathtools}

\newtheorem{theorem}{Theorem}[section]
\newtheorem{corollary}[theorem]{Corollary}
\newtheorem{lemma}[theorem]{Lemma}
\newtheorem{proposition}[theorem]{Proposition}

\theoremstyle{definition}

\theoremstyle{remark}
\newtheorem{remark}[theorem]{Remark}

\newcommand{\bC}{\mathbb C}

\newcommand{\vf}{\varphi}
\newcommand{\vfd}{\dot{\varphi}}

\global\long\def\oneb{\bar{1}}

\newcommand{\beq}{\begin{equation}}
\newcommand{\eeq}{\end{equation}}

\newcommand{\eps}{\epsilon}

\newcommand{\norm}{|\!|}
\newcommand{\bignorm}{\bigg|\!\bigg|}

\DeclareMathOperator{\im}{Im}
\DeclareMathOperator{\re}{Re}

\def\sideremark#1{\ifvmode\leavevmode\fi\vadjust{\vbox to0pt{\vss
 \hbox to 0pt{\hskip\hsize\hskip1em
 \vbox{\hsize2.7cm\tiny\raggedright\pretolerance10000
  \noindent #1\hfill}\hss}\vbox to8pt{\vfil}\vss}}}

\begin{document}
\title[]{Deformations and Embeddings of Three-dimensional Strictly Pseudoconvex CR Manifolds}
\author{Sean N. Curry}
\address{Department of Mathematics, Oklahoma State University, Stillwater, OK 74078}
\email{sean.curry@okstate.edu}
\author{Peter Ebenfelt}
\address{Department of Mathematics, University of California at San Diego, La Jolla, CA 92093-0112}
\email{pebenfel@math.ucsd.edu}

\date{\today}
\thanks{The second author was supported in part by the NSF grants DMS-1600701 and DMS-1900955.}

\begin{abstract}
Abstract deformations of the CR structure of a compact strictly pseudoconvex hypersurface $M$ in $\mathbb{C}^2$ are encoded by complex functions on $M$. In sharp contrast with the higher dimensional case, the natural integrability condition for $3$-dimensional CR structures is vacuous, and generic deformations of a compact strictly pseudoconvex hypersurface $M\subseteq \mathbb{C}^2$ are not embeddable even in $\mathbb{C}^N$ for any $N$. A fundamental (and difficult) problem is to characterize when a complex function on $M \subseteq \mathbb{C}^2$ gives rise to an actual deformation of $M$ inside $\mathbb{C}^2$. In this paper we study the embeddability of families of deformations of a given embedded CR $3$-manifold, and the structure of the space of embeddable CR structures on $S^3$. We show that the space of embeddable deformations of the standard CR $3$-sphere is a Frechet submanifold of $C^{\infty}(S^3,\mathbb{C})$ near the origin. We establish a modified version of the Cheng-Lee slice theorem in which we are able to characterize precisely the embeddable deformations in the slice (in terms of spherical harmonics). We also introduce a canonical family of embeddable deformations and corresponding embeddings starting with any infinitesimally embeddable deformation of the unit sphere in $\mathbb{C}^2$.
\end{abstract}

\maketitle

\section{Introduction and Main Results}

A fundamental problem in CR geometry is that of characterizing embeddability of abstract CR manifolds, where a CR manifold is said to be \emph{embeddable} if it is CR embeddable in $\mathbb{C}^N$ for some $N$. By the work of Boutet de Monvel and Kohn \cite{BoutetdeMonvel1975, Kohn1986}, embeddability of compact strictly pseudoconvex CR manifolds can be characterized in terms of a closed range property of $\bar{\partial}_b$. In particular, when the dimension of the CR manifold is at least $5$ it is always embeddable \cite{BoutetdeMonvel1975}. On the other hand, compact strictly pseudoconvex CR $3$-manifolds are generically not embeddable \cite{BurnsEpstein1990b}. The first known examples of such nonembeddable CR $3$-manifolds go back to Rossi \cite{Rossi1965} who showed that certain classical $\mathrm{SU}(2)$-invariant structures on $S^3$ are not embeddable (though, being real analytic, they are locally embeddable); a locally nonembeddable example was given by Nirenberg \cite{Nirenberg1974,Nirenberg1975}. (Nirenberg's example can be compactified to give a CR structure on $S^3$, and his construction already indicated that nonembeddability was generic in the compact case.) The question of embeddability of compact strictly pseudoconvex CR $3$-manifolds has continued to receive much attention, and many authors have sought to achieve a deeper understanding of the set of embeddable structures. Epstein \cite{Epstein1998,Epstein2012} has studied the set of embeddable deformations of a given compact embeddable CR structure in terms of index theory for the corresponding (relative) Szeg\H{o} projectors, and shown that the set of embeddable structures is closed in the $C^{\infty}$ topology \cite{Epstein2012}. Chanillo, Chiu and Yang \cite{ChanilloChiuYang2012,ChanilloChiuYang2013} have given a sufficient condition for embeddability in terms of CR Yamabe invariants; specifically they show that a compact CR structure is embeddable if it has positive Yamabe invariant and nonnegative CR Paneitz operator. A partial converse has recently been established by Takeuchi \cite{Takeuchi2019-arxiv} who showed that the CR Paneitz operator of an embeddable compact CR $3$-manifold is always nonnegative.

In this paper we study the embeddability of families of abstract deformations of a fixed compact strictly pseudoconvex CR $3$-manifold embedded in $\mathbb{C}^2$, and the structure of the space of embeddable deformations (as a subset of the space of all abstract deformations) of the standard CR $3$-sphere in $\mathbb{C}^2$. By the stability theorem of Lempert \cite{Lempert1994}, a small abstract deformation of a compact strictly pseudoconvex hypersurface in $\mathbb{C}^2$ is embeddable (in $\mathbb{C}^N$ for some $N$) if and only if it is embeddable in $\mathbb{C}^2$. We therefore restrict our attention to embeddability in $\mathbb{C}^2$. We shall mainly consider CR structures on the $3$-sphere $S^3$ near its standard CR structure, i.e. the strictly pseudoconvex CR structure that it inherits as the boundary of the unit ball in $\mathbb{C}^2$. Recall that a strictly pseudoconvex CR structure $(M,H,J)$ on a smooth $3$-manifold $M$ is a contact distribution $H\subseteq TM$ equipped with a bundle endomorphism $J:H\to H$ satisfying $J^2 = -\mathrm{id}$. When $M=S^3$, by a result of Eliashberg \cite{Eliashberg1990}, a CR structure can be embedded in $\mathbb{C}^2$ only if the underlying contact structure agrees with that of the standard CR sphere.
Let $\Gamma(\mathcal{J})$ denote the space of smooth positively oriented CR structures on $S^3$ compatible with its standard contact distribution $H$. Let $\Gamma(\mathcal{J})_{emb} \subset \Gamma(\mathcal{J})$ denote the subset of CR structures that are embeddable in $\mathbb{C}^2$.
In \cite{ChengLee1995} it is shown that $\Gamma(\mathcal{J})$ is a smooth tame Fr\'echet manifold in the sense of Hamilton \cite{Hamilton1982BAMS}, with respect to the scale of standard $L^2$-based Sobolev spaces on $M$. The same holds for the space of embeddable CR structures near the standard CR $3$-sphere:
\begin{theorem}[\cite{Bland1994,Lempert1997}]\label{thm:Frechet}	
$\Gamma(\mathcal{J})_{emb} \subset \Gamma(\mathcal{J})$ is a smooth tame Fr\'echet submanifold near the standard CR sphere.
\end{theorem}
To understand the embeddable CR structures on $S^3$ more concretely, we parametrize $\Gamma(\mathcal{J})$ by complex functions on $S^3$ in the following way. First, note that specifying a CR structure $J$ compatible with $H$ is the same as specifying its $\pm i$ eigenspaces $T^{1,0}$ and $T^{0,1}=\overline{T^{1,0}}$ as subbundles of $\mathbb{C}\otimes H$. Let $(z,w)$ denote the coordinates on $\mathbb{C}^2$ and define the following vector fields on $S^3$,
\beq\label{eqn:S3-Z1}
Z_1 = \bar{w}\frac{\partial}{\partial z} - \bar{z} \frac{\partial}{\partial w}, \qquad Z_{\oneb}= \overline{Z_1}
\eeq
spanning $T^{1,0}$ and $T^{0,1}$ respectively for the standard CR $3$-sphere $(S^3,H,J_0)$. A complex function $\varphi = \varphi_{1}{}^{\oneb}$ on $S^3$ with $\norm\varphi\norm_{\infty}<1$ defines an oriented CR structure on $(S^3,H)$ by defining its holomorphic tangent space $^{\vf}T^{1,0}$ to be spanned by
$$
Z^{\varphi}_1 = Z_1 + \varphi_1{}^{\oneb} Z_{\oneb}.
$$
(Up to complex conjugation, all CR structures compatible with $H$ are realized this way.) Strictly speaking, $\varphi$ should be interpreted as a section of $(T^{1,0})^*\otimes {T^{0,1}}$  and we refer to $\varphi$ as the \emph{deformation tensor}, though we usually trivialize $(T^{1,0})^*\otimes {T^{0,1}}$ using $Z_1$ and $Z_{\oneb}$ in order to think of $\varphi$ as a function. We let $\mathfrak{D}$ denote the space of smooth deformation tensors, and let $\mathfrak{D}_{emb}\subset \mathfrak{D}$ be the subset of deformations that are embeddable in $\mathbb{C}^2$. The main goal of this paper is to better understand the space of embeddable deformation tensors $\mathfrak{D}_{emb}$ on $S^3$, thought of as a space of functions using the standard frame $Z_1$, $Z_{\oneb}$.

In \cite{BurnsEpstein1990b} Burns and Epstein showed that there is an infinite dimensional linear space within the space of embeddable deformation tensors $\mathfrak{D}_{emb}$ near the origin (i.e. the trivial deformation corresponding to the standard structure on $S^3$), characterized by the vanishing of certain terms in the spherical harmonic decomposition. To make this more precise we introduce the spherical harmonic spaces $H_{p,q}$ of functions on $S^3$ that are the restrictions of harmonic homogeneous polynomials of bidegree $(p,q)$ on $\mathbb{C}^2$ for each $p,q\geq 0$.
We denote the component of $\varphi$ in $H_{p,q}$ by $\varphi_{p,q}$, so that the $L^2$ orthogonal spherical harmonic decomposition of $\vf$ is given by $\varphi = \sum_{p,q} \varphi_{p,q}$. Define $\mathfrak{D}_{BE}\subset \mathfrak{D}$ to be the set of all deformation tensors $\varphi$ such that $\varphi_{p,q}=0$ if $q < p + 4$ (our deformation tensor is the conjugate of Burns and Epstein's). Burns and Epstein showed that if $\varphi \in \mathfrak{D}_{BE}$ is sufficiently small in $C^4$ then the deformation is embeddable. This has a clear conceptual explanation given by Bland \cite{Bland1994} in terms of Lempert's theory of extremal discs for the Kobayashi metric, the corresponding circular  representation, and nonnegativity of the Fourier coefficients of the conjugated deformation tensor $\overline{\vf}$ (relative to an $S^1$-invariant frame); cf.\ \cite{BlandDuchamp1991,Lempert1992, Lempert1994, Patrizio1987}. Examining the linearized action of the contact diffeomorphisms on the space of CR structures on $S^3$ suggests that the space of Burns-Epstein deformations (or more precisely a certain subspace of the Burns-Epstein deformations satisfying an additional condition along the critical diagonal $p = q + 4$) should give a slice for the action of the group of contact diffeomorphisms on the space of embeddable CR structures. But this has not been fully resolved in the literature; in particular, such a result has not been established in the $C^{\infty}$ case. One of our main results is a slice theorem for the $C^{\infty}$ embeddable CR structures on $S^3$ near its standard CR structure, see Theorem \ref{thm:slice}. To do this we first prove a modified version of the Cheng-Lee slice theorem \cite{ChengLee1995} for the space of abstract deformations of the standard CR structure on $S^3$, and then show that restricting to a natural subspace of the modified slice gives a slice theorem for the embeddable CR structures.

Before stating our slice theorems we briefly discuss the corresponding linearized problem. Given any CR hypersurface $M \subseteq \mathbb{C}^2$, the infinitesimally embeddable abstract deformations may be understood concretely as follows. Let $M_t$ be any smooth $1$-parameter family of strictly pseudoconvex hypersurfaces in $\mathbb{C}^2$ with $M_0=M$, defined as the zero loci of a smooth family of defining functions $\rho_t$. It is always possible to find a family of contact diffeomorphisms $\psi_t:M\to M_t$ with $\psi_0 =\mathrm{id}$ parametrizing the family $M_t$. Using $\psi_t$ one may pull back the CR structures of the $M_t$ to $M$ in order to obtain a family of CR structures on $M$ whose holomorphic tangent spaces are spanned by $Z_1^t = Z^0_1 + \vf_1{}^{\oneb}(t)Z^0_{\oneb}$ where $Z^0_1$ is
a (unitary) frame for the holomorphic tangent space of $M=M_0$. For purely aesthetic reasons, we lower the index $\oneb$ on $\vf_1{}^{\oneb}(t)$ using the Levi form of $\rho_t$ to obtain $\vf_{11}(t)$. A straightforward geometric calculation shows that if $\vfd = \vfd_{11} = \left.\frac{d}{dt}\right|_{t=0} \vf_{11}(t)$ then
\beq\label{eq:inf-emb-eqn}
\vfd_{11} = (\nabla_1\nabla_1 + iA_{11})f
\eeq
for some function $f$ where $\mathrm{Re}\,f = -\dot{\rho} =-\left.\frac{d}{dt}\right|_{t=0} \rho_t |_{M}$ is the normal velocity of the deformation at $t=0$ (see, e.g, \cite{BlandEpstein1996,Cheng2014,CE2018-obstruction-flatI,HirachiMarugameMatsumoto2017} or Lemma \ref{lem:embed} below); here $\nabla$ is the Tanaka-Webster connection of the contact form $i\partial \rho_0 |_M$ and $A_{11}$ is the corresponding pseudohermitian torsion. In the case of the standard CR sphere, defined by $\rho_0=1-|z|^2-|w|^2$, \eqref{eq:inf-emb-eqn} simply becomes
\beq
\vfd_{11} = Z_1 Z_1 f.
\eeq
The space $\mathfrak{D}_0$ of infinitesimally embeddable deformation tensors on $S^3$ is easily understood using spherical harmonics. The vector field $Z_1$ sends each $H_{p,q}$ isomorphically onto $H_{p-1,q+1}$ unless $p=0$ in which case $Z_1$ is zero. It follows that $\vfd$ is an embeddable infinitesimal deformation (i.e.\ $\vfd$ is in the range of $Z_1Z_1$) if and only if $\vfd_{p,q}=0$ for $q=0,1$. Its easy to see that every infinitesimally embeddable deformation tensor $\vfd$ can be realized as $\left.\frac{d}{dt}\right|_{t=0}\vf(t)$ for some family of embeddable deformation tensors $\vf(t)$. We therefore sometimes refer to $\vfd\in\mathcal{D}_0$ as a \emph{linearized embeddable deformation}. 

Having understood the embeddability problem at the linearized (i.e.\ infinitesimal) level about the standard CR 3-sphere, it is natural to ask if we can similarly characterize the embeddable deformations. Such a characterization is possible if we work modulo contact diffeomorphisms. Let $\mathcal{C}$ denote the space of contact diffeomorphisms on $S^3$. The Lie algebra of $\mathcal{C}$ is the space of contact (Hamiltonian) vector fields, which can be identified with $C^{\infty}(S^3,\mathbb{R})$ once a contact form on $S^3$ has been chosen (we always take the standard contact form $\theta = i(zd\bar{z} + wd\bar{w})$ on $S^3$ which normalizes the Levi form to be $h_{1\oneb} =1$ in the frame $Z_1$). The linearization of the natural action $\mathcal{C}\times \Gamma(\mathcal{J}) \to \Gamma(\mathcal{J})$ at $(\mathrm{id},J_{0})$ is
\beq\label{eq:linearization-cont-diff}
(g,\vfd_{11}) \mapsto \vfd_{11} + iZ_1Z_1 g,
\eeq
where $g \in C^{\infty}(S^3,\mathbb{R})$ is the potential for a contact (Hamiltonian) vector field and $\vfd$ is a deformation tensor on $S^3$ (here we are identifying $\mathfrak{D}$ with the tangent space of $\Gamma(\mathcal{J})$ at $J_0$). As an immediate consequence of \eqref{eq:linearization-cont-diff}, it was observed in Burns-Epstein \cite{BurnsEpstein1990b} that an infinitesimal slice for the action of the contact diffeomorphisms on CR structures at $J_0$ is given by $\mathfrak{D}'_{BE}\oplus \mathfrak{D}_0^{\perp}$ where $\mathfrak{D}_0^{\perp} \subseteq \mathfrak{D}$ is the $L^2$ orthogonal complement to $\mathfrak{D}_0$ and $\mathfrak{D}'_{BE} \subseteq \mathfrak{D}_{BE}$ is the subspace of all $\varphi\in \mathfrak{D}_{BE}$ that additionally satisfy the reality condition $\im\,((Z_{\oneb})^2 \varphi_{p,p+4}) = 0$ along the critical diagonal. (The latter reality condition is equivalent to saying that $\vf$ must be $L^2$ orthogonal to the image of the real $S^1$-invariant functions on $S^3$ under $i(Z_1)^2$, where the inner product is the real part of the complex inner product.)

In Cheng-Lee \cite{ChengLee1995} it was shown that the space of marked CR structures on $S^3$ near the standard CR structure can be locally identified with $\mathcal{C}\times \mathcal{S}$  where $\mathcal{S}$ is the set of all deformation tensors $\varphi$ such that $\im\,(Z_{\oneb}Z_{\oneb} \varphi) = 0$. Marking here refers to the choice of a point in the CR Cartan frame bundle \cite{CapSlovak2009,ChernMoser1974,Tanaka1962} of the given CR structure on $(S^3,H)$; the symmetry group of any marked CR structure is trivial, so working with marked structures eliminates the need to try to mod out by the noncompact symmetry group of the standard CR sphere. For our purposes, we need a modified version of the Cheng-Lee slice theorem which uses the linearly equivalent slice $\mathfrak{D}'_{BE}\oplus \mathfrak{D}_0^{\perp}$. Let $\Gamma(\mathcal{J})^m$ denote the space of marked CR structures on $(S^3,H)$, which we identify with the space $\mathfrak{D}^m$ of marked deformations of $(S^3,H,J_0)$. The contact diffeomorphisms act naturally by pullback on $\Gamma(\mathcal{J})$; a contact diffeomorphism that takes $J_1\in\Gamma(\mathcal{J})$ to $J_2\in\Gamma(\mathcal{J})$ lifts to an equivariant diffeomorphism between the corresponding CR Cartan frame bundles \cite{Cap2008,CapSlovak2009} so that the action of contact diffeomorphisms on the CR structures $\Gamma(\mathcal{J})$ extends naturally to an action on the marked CR structures $\Gamma(\mathcal{J})^m$ (see, e.g., \cite{ChengLee1995}), and hence on $\mathfrak{D}^m$ by identification with $\Gamma(\mathcal{J})^m$.
\begin{theorem}\label{thm:CL-slice-mod}
Fix any marking $y_0$ of the standard CR sphere.  Then
\begin{itemize}
\item[(i)] The natural action $\mathcal{C}\times \mathfrak{D}^m  \to \mathfrak{D}^m$ restricts to a local smooth tame diffeomorphism $P:\mathcal{C}\times (\mathfrak{D}'_{BE} \oplus \mathfrak{D}_0^{\perp}) \times \{y_0\} \to \mathfrak{D}^m$ in a neighborhood of $(0,y_0)\in \mathfrak{D}^m$;
\item[(ii)] For $\Psi\in \mathcal{C}$ sufficiently near the identity, the image of $(\mathfrak{D}'_{BE} \oplus \mathfrak{D}_0^{\perp})\times \{y_0\}$ under $\Psi$ is disjoint from itself unless $\Psi=\mathrm{Id}$.
\end{itemize}
\end{theorem}
The proof of this modified Cheng-Lee slice theorem can be obtained by adapting the proof of Theorem B in \cite{ChengLee1995}. For the reader's convenience we provide a slightly simplified proof of this theorem in Section \ref{sec:proofs-slice-thms}. The advantage of this modified slice theorem is that a linear subspace of the slice gives a slice for the embeddable deformations. Let $\mathfrak{D}^m_{emb}$ denote the space of marked embeddable deformations of the standard CR sphere. We shall prove the following slice theorem for the set of embeddable deformations, which also immediately implies Theorem \ref{thm:Frechet}.
\begin{theorem}\label{thm:slice}
Fix any marking $y_0$ of the standard CR sphere. Then
\begin{itemize}
\item[(i)] The natural action $\mathcal{C}\times \mathfrak{D}^m \to \mathfrak{D}^m$ restricts to a local smooth tame immersion $P_{emb}:\mathcal{C}\times \mathfrak{D}'_{BE} \times \{y_0\} \to \mathfrak{D}^m$ in a neighborhood of $(\mathrm{id},0)\in \mathcal{C}\times \mathfrak{D}'_{BE}$ whose image is a neighborhood of $(0,y_0)$ in $\mathfrak{D}^m_{emb}$;
\item[(ii)] For $\Psi\in \mathcal{C}$ sufficiently near the identity, the image of $\mathfrak{D}'_{BE}\times \{y_0\}$ under $\Psi$ is disjoint from itself unless $\Psi=\mathrm{Id}$.
\end{itemize}
\end{theorem}
We observe that by Theorem \ref{thm:CL-slice-mod}, Theorem \ref{thm:slice} is equivalent to the statement that $(\mathfrak{D}'_{BE} \oplus \mathfrak{D}_0^{\perp} ) \setminus \mathfrak{D}'_{BE}$ consists solely of nonembeddable deformations (near the origin). This question was considered in \cite{BurnsEpstein1990b} where it was shown that the nonembeddable deformations form a $G_{\delta}$-set in $(\mathfrak{D}'_{BE} \oplus \mathfrak{D}_0^{\perp} ) \setminus \mathfrak{D}'_{BE}$; the results in \cite{Epstein1998} imply that this $G_{\delta}$-set is open. Theorem \ref{thm:slice} settles the question completely; a sufficiently small $\vf \in \mathfrak{D}'_{BE} \oplus \mathfrak{D}_0^{\perp}$ is embeddable if and only if $\vf \in \mathfrak{D}'_{BE}$.

Another consequence of Theorem \ref{thm:slice} is a normal form for embeddable CR structures, unique up to an action of $\mathrm{Aut}(S^3)$ on $\mathfrak{D}'_{BE}$.
\begin{corollary}\label{cor:BE-normalization}
For sufficiently small deformations $\varphi$ of the standard CR sphere, $\varphi$ is an embeddable deformation if and only if there exists a smooth contact diffeomorphism such that the pulled back CR structure corresponds to a deformation $\tilde{\vf} \in \mathfrak{D}'_{BE}$.
\end{corollary}
Note that $\mathfrak{D}'_{BE}$ can be replaced by $\mathfrak{D}_{BE}$ in Corollary \ref{cor:BE-normalization} (at the expense of leaving also the freedom to act by an $S^1$-equivariant contact diffeomorphism on $\mathfrak{D}_{BE}$). Such a result was only previously known in finite regularity, with the notion of ``sufficiently small'' depending on the regularity; see the work of Bland and Bland-Duchamp \cite{Bland1994, BlandDuchamp2011, BlandDuchamp2014}.

The above characterization of embeddable deformations is satisfying, but it does not really give a practical means of checking for embeddability since one must first normalize the deformation tensor by an appropriate contact diffeomorphism. We would like to say something about the embeddability of a deformation without the need to first normalize it. At the linearized (i.e.\ infinitesimal) level this is clear, as explained above. To what extent does a similar characterization of embeddability hold beyond the linear level? By taking a completely different approach to the problem using geometric flows we provide the following result describing embeddable structures without the need to normalize by contact diffeomorphisms.
\begin{theorem}\label{thm:vft}
For $\vfd \in \mathfrak{D}_0$ sufficiently small there exists a smooth family $\vf(t)\in \mathfrak{D}_{emb}$ such that $\vf(t) = t\vfd + \mu(t)$ for $t\in[0,1]$, where $\mu(t)=O(t^2)$ and $\mu(t)\in \mathfrak{D}_0^{\perp}$. Moreover, there exists a smooth family of embeddings $\Phi_t:S^3\to\mathbb{C}^2$, with $\Phi_0=\mathrm{Id}$, realizing the deformation $\varphi(t)$ for each $t\in [0,1]$.
\end{theorem}
\begin{remark}
We make two remarks.
\begin{itemize}
\item[(1)] In fact, it follows from the more detailed version of the theorem, Theorem \ref{thm:smooth-soln}, that the family $\vf(t)$ is canonical up to the choice of a constant $\lambda$ that determines the initial velocity of the family of embeddings $\Phi_t$ (setting $\lambda=-1$ ensures that the family of surfaces $\Phi_t(S^3)$ move outward and locally foliate $\mathbb{C}^2$); the resulting time one map $\mathfrak{D}_0 \to \mathfrak{D}_{emb}$ taking $\vfd$ to $\vf(1)$ has a linearization at $\vfd=0$ which is the identity and hence can be thought of as an exponential map.
\item[(2)] Note that, in terms of spherical harmonics, the condition $\psi\in  \mathfrak{D}_0^{\perp}$ means that $\psi_{p,q}=0$ except possibly when $q=0,1$. Also, note that $\psi(t)$ will not be zero in general (as can be seen by an inspection of the proof of Proposition \ref{prop:formalSol} for the special case where, say, $\vfd \in H_{p,2}$).
\end{itemize}
\end{remark}
In the special case where $\vfd \in \mathfrak{D}_{BE}$ we naturally find that $\vf(t)=t\vfd$ (i.e. $\psi(t)=0$) and we obtain analyticity of $\Phi_t$ in $t$.
More precisely:
\begin{theorem}\label{thm:BE-parametric}
For $\vfd \in \mathfrak{D}_{BE}$ sufficiently small there exists a family $\Phi_{\tau}:S^3\to\mathbb{C}^2$ with complex parameter $\tau$, such that for each $\tau$, $|\tau|<2$,
\begin{itemize}
\item[(i)]  $\Phi_{\tau}$ is a smooth embedding which realizes the deformation $\vf(\tau)=\tau\vfd$;
\item [(ii)] $\Phi_{\tau}$ is analytic in $\tau$ as a function with values in the Banach space of $C^k$ maps $S^3\to\mathbb{C}^2$ for any $k$.
\end{itemize}
\end{theorem}
Note that this recovers the result of Burns-Epstein \cite[Theorem 5.3]{BurnsEpstein1990b} by setting $\tau=1$.

As a by-product of our approach, we also establish the embeddability of a family of deformations of an embedded structure that satisfy a well known necessary condition (stated for $t=0$ above), under a natural additional condition that forces the resulting family of embeddings to move inwards (or outwards).
As mentioned above, given a CR 3-manifold $(M,H,J)$ embedded in a complex surface, an infinitesimal deformation tensor $\vfd$ will be infinitesimally embeddable if and only if it satisfies \eqref{eq:inf-emb-eqn} for some complex function $f$ on $M$. Given a family of CR hypersurfaces $M_t\subseteq \mathbb{C}^2$ with $M_0=M$ contact parametrized by $\psi_t:M\to M_t$, with $\psi_0=\mathrm{id}$, \eqref{eq:inf-emb-eqn} applies at each time $t$ on $M_t$. Pulling back using $\psi_t:M\to M_t$ we obtain a family of embeddable deformations $\vf(t)$ with $\vf(0)=0$ and a family of complex functions $f_t$ on $M$ satisfying a certain second order equation at each time $t$, corresponding to \eqref{eq:inf-emb-eqn}; $f_t$ can be interpreted as the complex normal component of the variational vector field $\dot{\psi}_t$ arising from the family of embeddings $\psi_t$ (more precisely, as the corresponding function on $M$). Saying that a family of abstract deformations $\vf(t)$ satisfies this condition (for some family $f_t$) in principle says that the deformation $\vf(t)$ moves tangent to the space of embeddable deformations at each time $t$. Borrowing terminology from Jih-Hsin Cheng \cite{Cheng2014} we will refer to this condition on $\varphi(t)$ as the \emph{tangency condition}; we shall also refer to the family $f_t$ as a \emph{family of potentials} corresponding to $\varphi(t)$. For the precise formulation of the tangency condition see Section \ref{sec:embeddings-from-solutions} and also Lemma \ref{lem:embed}. Given a family of deformations of an embeddable CR structure satisfying the tangency condition, it is natural to ask whether this family is embeddable. Our result is the following:
\begin{theorem}\label{thm:tangency-cond}
Let $M$ be a compact strictly pseudoconvex hypersurface in $\mathbb{C}^2$ and let $\vf(t)$ be a 1-parameter family of deformations of the induced CR structure on $M$ with $\vf(0)=0$. Suppose $\vf(t)$ satisfies the tangency condition for all $t$ with a family of potentials $f_t$ with $\re f_t$ having strict sign. Then there exists $\epsilon>0$ such that $\varphi(t)$ is an embeddable deformation for all $t\in [0,\epsilon)$.
\end{theorem}

For a more precise statement see Theorem \ref{P:embed} below. This generalizes Cheng's theorem for fillable structures \cite[Theorem A]{Cheng2014} to the case of embeddable structures and allows $\re f_t$ to have either sign. We remark that the sufficient condition in Theorem \ref{thm:tangency-cond} for embeddability of the family $\vf(t)$ is also a necessary condition. This follows because a family of CR hypersurfaces $M_t\subseteq \mathbb{C}^2$ ($t\in [0,\epsilon]$) can always be re-embedded by applying a smooth $t$-dependent family of dilations such that it moves to the pseudoconvex side for all time (or, if desired instead, to the pseudconcave side for all time) so that the family of potentials $f_t$ then satisfies $\re f_t>0$ for all $t$ (or, if desired instead, $\re f_t<0$ for all $t$). See Section \ref{sec:emb-def}.

This paper is organized as follows. In Section \ref{sec:deformations} we give some preliminaries on deformations of $3$-dimensional CR structures and introduce the tangency equation for families of embeddable deformations, which makes precise the tangency condition referred to in Theorem \ref{thm:tangency-cond}. In Section \ref{sec:embeddings-from-solutions} we explain how one obtains embeddings from solutions to the tangency equation and establish
Theorem \ref{P:embed}, which implies Theorem \ref{thm:tangency-cond}. In Section \ref{sec:solutions-tangency-eqn} we study the solvability of the tangency equation for small deformations of the standard CR $3$-sphere, and establish Theorems \ref{thm:vft} and  \ref{thm:BE-parametric}. Finally, in Section \ref{sec:proofs-slice-thms} we prove the slice theorems, Theorem \ref{thm:CL-slice-mod} and Theorem \ref{thm:slice}. We remark that the main sections, Sections \ref{sec:embeddings-from-solutions}-\ref{sec:proofs-slice-thms}, are largely independent from each other. Section \ref{sec:embeddings-from-solutions} is primarily geometric, and makes use of the Fefferman ambient metric construction in the framework of Hirachi-Marugame-Matsumoto \cite{HirachiMarugameMatsumoto2017}. In both Section \ref{sec:solutions-tangency-eqn} and Section \ref{sec:proofs-slice-thms} we make use of the Nash-Moser inverse function theorem as presented in Hamilton \cite{Hamilton1982BAMS} (see also Cheng-Lee \cite{ChengLee1995} for a brief introduction to this and Hamilton's tame Fr\'echet category). These sections also make use of an elliptic regularity argument adapted from \cite{BurnsEpstein1990b} which appears first in the proof of Proposition \ref{prop:convergence-ft}. The proof of Theorem \ref{thm:smooth-soln} (a more precise version of Theorem \ref{thm:vft}) uses arguments from the theory of parabolic evolution equations.

\subsection*{Acknowledgments}
The authors would like to thank the anonymous referee for helpful comments leading to a significant improvement of the exposition.

\section{Deformations of $3$-Dimensional CR Structures}\label{sec:deformations}


Let $M$ be a smooth oriented $3$-manifold. A \emph{contact structure} on $M$ is a rank $2$ subbundle $H\subset TM$ which is nondegenerate in the sense that if $H$ is locally given as the kernel of some $1$-form $\theta$, then $\theta\wedge d\theta$ is nowhere vanishing.
A \emph{CR structure} on $(M,H)$ is given by a smooth endomorphism $J:H\to H$ such that $J^2=-\mathrm{id}$. We refer to $(M,H,J)$ as a \emph{strictly pseudoconvex CR $3$-manifold}. The partial complex structure $J$ on $H\subset TM$ defines an orientation of $H$, and therefore defines an orientation on the annihilator subbundle $H^{\perp}:=\mathrm{Ann}(H)\subset T^*M$. A nowhere vanishing section $\theta$ of $H^{\perp}$ is called a \emph{contact form} for $H$. A contact form $\theta$ is positively oriented if $d \theta|_{H}$ is compatible with the orientation of $H$, equivalently, if $d\theta( \,\cdot\, , J\,\cdot\,)$ is positive definite on $H$. A CR structure $(M,H,J)$ together with a choice of positively oriented contact form $\theta$ is referred to as a \emph{pseudohermitian structure} \cite{Tanaka1975,Webster1978}. The \emph{Reeb vector field} of a contact form $\theta$ is the vector field $T$ uniquely determined by $\theta(T)=1$ and $d\theta(T,\,\cdot\,) =0$.

Given a CR manifold $(M,H,J)$ we decompose the complexified contact distribution $\mathbb{C}\otimes H$
as $T^{1,0}\oplus T^{0,1}$, where $J$ acts by $i$ on $T^{1,0}$ and by $-i$ on $T^{0,1}=\overline{T^{1,0}}$. Let $\theta$ be a positively oriented contact form on $M$. Let $Z_1$ be a local frame for the \emph{holomorphic tangent bundle} $T^{1,0}$ and $Z_{\oneb}=\overline{Z_1}$, so that $\{T,Z_1,Z_{\oneb}\}$ is a local frame for $\mathbb{C}\otimes TM$. Then the dual frame $\{\theta,\theta^1,\theta^{\oneb}\}$ is referred to as an \emph{admissible coframe} and one has
\begin{equation}\label{eqn:h11-definition}
d \theta = i h_{1\oneb} \theta^1 \wedge \theta^{\oneb}
\end{equation}
for some positive smooth function $h_{1\oneb}$. The function $h_{1\oneb}$ is the component of the \emph{Levi form} $\mathrm{L}_{\theta}(U,\overline{V})=-id\theta(U,\overline{V})$ on $T^{1,0}$, that is
\begin{equation*}
\mathrm{L}_{\theta}(U^1Z_1,V^{\oneb}Z_{\oneb}) = h_{1\oneb}U^1V^{\oneb}.
\end{equation*}
It is convenient to scale $Z_1$ so that $h_{1\oneb}=1$, and we will typically do so. In any case, we write $h^{1\oneb}$ for the multiplicative inverse of $h_{1\oneb}$. The Tanaka-Webster connection associated to $\theta$ is given in terms of such a local frame $\{T,Z_1,Z_{\oneb}\}$ by
\begin{equation*}
\nabla Z_1 = \omega_1{}^{1}\otimes Z_1, \quad \nabla Z_{\oneb} = \omega_{\oneb}{}^{\oneb}\otimes Z_{\oneb}, \quad \nabla T =0
\end{equation*}
where the connection $1$-forms $\omega_1{}^{1}$ and $\omega_{\oneb}{}^{\oneb}$ satisfy
\begin{equation}\label{eqn:pseudohermitian-connection1}
d \theta^1 = \theta^1\wedge \omega_1{}^{1} + A^1{}_{\oneb}\,\theta\wedge\theta^{\oneb}, \text{ and}
\end{equation}
\begin{equation}\label{eqn:pseudohermitian-connection2}
\omega_1{}^{1} + \omega_{\oneb}{}^{\oneb} =h^{1\oneb} d h_{1\oneb},
\end{equation}
for some function $A^1{}_{\oneb}$. The uniquely determined function $A^1{}_{\oneb}$ is known as the \emph{pseudohermitian torsion}. Components of covariant derivatives will be denoted by adding $\nabla$ with an appropriate subscript, so, e.g., if $u$ is a function then $\nabla_1 u = Z_1 u$, $\nabla_1\nabla_1 u = Z_1 Z_1 u - \omega_1{}^{1}(Z_1)Z_1u$ and $\nabla_0\nabla_1 u = T Z_1 u - \omega_1{}^{1}(T)Z_1u$. We may also use $h_{1\oneb}$ and $h^{1\oneb}$ to raise and lower indices, so that $A_{\oneb\oneb}= h_{1\oneb}A^1{}_{\oneb}$ and $A_{11}=h_{1\oneb}A^{\oneb}{}_{1}$, with $A^{\oneb}{}_{1} =\overline{A^1{}_{\oneb}}$. Note that when $h_{1\oneb}=1$ raising and lowering indices is a trivial operation.

\subsection{Abstract Deformations}

Let $(M,H,J)$ be a compact, strictly pseudoconvex, three-dimensional CR manifold. Consider a smooth family of CR structures $(M, H_t, J_t)$ on $M$ with $(H_0,J_0)=(H,J)$. By Gray's theorem \cite{Gray1959} this family may be pulled back by a smooth family of diffeomorphisms to a family of the form $(M, H, \tilde{J}_t)$. When considering families of CR structures on $M$ we therefore always keep the contact distribution $H$ fixed. If $Z_1$ is holomorphic tangent vector field on $(M,H,J)$ then this amounts to requiring that the holomorphic tangent space of our deformed structure is spanned by a vector field of the form $Z_1 + \vf_{1}{}^{\oneb}Z_{\oneb}$ for some complex function $\vf = \vf_1{}^{\oneb}$ with $|\vf|^2 < 1$ on $M$. We shall fix a contact form $\theta$ on $M$ such that $Z_1$ is unitary (i.e. $h_{1\oneb}=1$ with respect to $Z_1$). Given a deformed CR structure spanned by $Z_1 + \vf_{1}{}^{\oneb}Z_{\oneb}$ we will always work with the normalized frame
\beq
Z_1^{\vf} = \frac{1}{\sqrt{1-|\vf|^2}}\left(Z_1 + \vf_{1}{}^{\oneb}Z_{\oneb}\right)
\eeq
so that the Levi form of $\theta$ with respect to the  deformed structure has component $h_{1\oneb}^{\vf}=1$. Given a family $(M, H, J_t)$ of CR structures on $M$, we may describe the deformation $J_t$ by a deformation tensor $\vf_1{}^{\bar 1}(t)$ via
\beq\label{Zdef}
Z_1^t:=\frac{1}{\sqrt{1-|\vf|^2}}\left(Z_1+\vf_{1}{}^{\bar 1}(t)Z_{\bar 1}\right),
\eeq
where we use the shorthand notation
\beq
|\vf|^2=|\vf_1{}^{\bar 1}(t)|^2.
\eeq
The corresponding admissible coframe $(\theta, \theta^1_t,\theta^{\oneb}_t)$ is obtained by choosing
\beq\label{th1def}
\theta^1_t:=\frac{1}{\sqrt{1-|\vf|^2}}\,\left (\theta^1-\vf_{\bar 1}{}^1(t)\, \theta^{\bar 1}\right ).
\eeq
Note that $J_t$ is easily recovered by writing $J_t = iZ_1^t\otimes \theta^1_t - i Z_{\oneb}^t \otimes \theta^{\oneb}_t$.
It is useful to invert the transformations $(Z_1, Z_{\bar 1})\mapsto (Z^t_1, Z^t_{\bar 1})$ and $(\theta^1,\theta^{\bar 1})\mapsto (\theta^1_t,\theta^{\bar 1}_t)$:
\begin{equation}
\label{ttransinv}
\begin{aligned}
Z_1 &=\frac{1}{\sqrt{1-|\vf|^2}}\,\left (Z^t_1-\vf_1{}^{\bar 1}(t)Z^t_{\bar 1}\right ),\\
\theta^1 &= \frac{1}{\sqrt{1-|\vf|^2}}\, \left(\theta^1_t+\vf_{\bar 1}{}^1(t)\theta^{\bar 1}_t\right).
\end{aligned}
\end{equation}
We denote by $\nabla^t$ the Tanaka-Webster connection of the pseudohermitian structure $(M,H,J_t,\theta)$, and by $A_{11}(t)$ its pseudohermitian torsion in the coframe $(\theta, \theta^1_t,\theta^{\oneb}_t)$. For the connection form $\omega_1{}^1$ on $M$ relative to the admissible coframe $(\theta,\theta^1,\theta^{\bar1})$, we shall write
\beq\label{connform}
\omega_1{}^1=\omega_1{}^1{}_1\,\theta^1+\omega_1{}^1{}_{\bar 1}\,\theta^{\bar 1}+\omega_1{}^1{}_0\,\theta,
\eeq
and similarly for the connection forms $\omega_1{}^1(t)$ of $\nabla^t$,
\beq\label{connform-t}
\omega_1{}^1(t)=\omega_1{}^1{}_1(t)\,\theta^1_t+\omega_1{}^1{}_{\bar 1}(t)\,\theta^{\bar 1}_t+\omega_1{}^1{}_0(t)\,\theta.
\eeq
Note that we then have, for a smooth function $f$,
\beq\label{defeq}
\nabla^t_{1}\nabla^t_{1} f=(Z^t_1)^2 f - \omega_1{}^1{}_1(t)Z^t_1f.
\eeq

\subsection{Embedded Deformations}\label{sec:emb-def}
Let $M$ be a strictly pseudoconvex hypersurface in a complex surface $X$. Then $M$ carries an \emph{induced CR structure} $(M,H,J)$, where $H_p$ is the maximal complex subspace in $T_p M\subset T_p X$ for each $p\in M$ and $J$ is induced from the standard complex structure on $X$.
We say that a smooth family of embeddings $\psi_t:M\to X$, $t\in[0,\epsilon)$, is a \emph{parametrized deformation} of $M$ if $\psi_0 = \mathrm{id}_M$ and $M_t=\psi_t(M)$ is strictly pseudoconvex for all $t$. 
By pulling back the induced CR stuctures on $M_t$ by $\psi_t$ for each $t$, one obtains a smooth family of CR structures $(M,H_t,J_t)$ on $M$ with $(M,H_0,J_0)=(M,H,J)$. 
We say that a parametrized deformation $\psi_t$ of $M$ is \emph{contact parametrized} if the induced family of CR structures $(M,H_t,J_t)$ on $M$ satisfies $H_t=H$ for all $t$, equivalently if $\psi_t:M\to M_t$ is a contact diffeomorphism for all $t$, where the contact structure on $M_t\subset X$ comes from the induced CR structure. 
By Gray's stability theorem \cite{Gray1959} any parametrized deformation may be reparametrized by a $1$-parameter family of diffeomorphisms of $M$ so that it becomes a contact parametrized deformation.

\begin{lemma}\label{lem:embed} 
Let $\psi_t:M\to X$, $t\in[0,\epsilon)$, be a contact parametrized deformation of the strictly pseudoconvex hypersurface $M\subset X$. Let $(M,H,J_t)$ denote the corresponding family of CR structures on $M$ and let $(\theta, \theta^1_t,\theta^{\oneb}_t)$ be a corresponding family of admissible coframes with deformation tensor $\vf_{11}(t)$ as above. Then there is a smooth family of functions $f_t\in C^{\infty}(M,\mathbb{C})$ such that
\beq\label{eqn:psi-dot-from-ft}
\dot{\psi}_t = {J\psi_t}_*\re\left(f_t T\right) - {\psi_t}_*\re\left( if_t T
+2(\nabla^{\oneb}_t f_t) Z^t_{\oneb}\right),
\eeq
and 
\beq\label{Nabla2ft}
\nabla^t_{1}\nabla^t_{1}\, f_t+iA_{11}(t)f_t=\frac{\dot\vf_{11}(t)}{1-|\vf(t)|^2},
\eeq 
where $\dot{\psi}_t = \frac{d}{dt} \psi_t$, $\dot \vf_{11}(t)=\frac{d}{dt}\, \vf_{11}(t)$ and $J$ denotes the complex structure on $X$. 
\end{lemma}
\begin{remark}
(i) Lemma \ref{lem:embed} is proved, e.g., in \cite[Lemma 4.5]{CE2018-obstruction-flatI} for the case of $t=0$ (note that what we are denoting by $f_0$ corresponds to $i\bar{f}$ in \cite{CE2018-obstruction-flatI}) and the generalization to arbitrary times $t$ is straightforward, though we include the full proof below for completeness. For other presentations, see \cite{BlandEpstein1996,Cheng2014,HirachiMarugameMatsumoto2017}. (ii) The denominator in the right hand side of \eqref{Nabla2ft} will arise naturally in the proof below, but it can also be explained by comparing the above lemma with \cite[Lemma 4.5]{CE2018-obstruction-flatI} and noting that by \eqref{Zdef} and \eqref{ttransinv} we have
\beq
Z_1^{t+s}=(1+O(s))Z^t_1+\left (\frac{\dot\vf_1{}^{\bar 1}(t)}{1-|\vf(t)|^2}\,s+O(s^2)\right)Z^t_{\bar 1}
\eeq
so that the infinitesimal deformation tensor for the family of CR structures $(M,H,J_t)$ at time $t$, relative to the frame $(Z_1^t, Z_{\oneb}^t)$, is $\frac{\dot\vf_1{}^{\bar 1}(t)}{1-|\vf(t)|^2}$.  (iii) Note that, for each $t$, the vector field ${\psi_t}_*\re\left( if_t T +2(\nabla^{\oneb}_t f_t) Z^t_{\oneb}\right)$ is tangent to $M_t = \psi_t(M)$. Hence, with our current conventions, if $f_t$ is pure imaginary for all $t$ then $\psi_t (M) = M$ for all $t$ and $\psi_t:M\to M$ is a family of contact diffeomorphisms of $(M,H)$. In general, the complex function $if_t$ is a kind of generalized Hamiltonian potential for the infinitesimal motion of $M_t$ in $X$; see Section 4 of \cite{CE2018-obstruction-flatI} for a more detailed discussion. (iv) Note that ${J\psi_t}_*T$ points to the pseudoconvex side of $M_t$ for each $t$. Hence, in particular, if $M_0$ bounds a strictly pseudoconvex domain and $\re f_t>0$ for all $t$, then $M_t$ moves inward (i.e.\ to the pseudoconvex side) as $t$ increases; on the other hand, if $\re f_t<0$ then $M_t$ moves outward. (v) Note that the family of functions $f_t$ depend only on the choice of contact form $\theta$, and not on the choice of coframing; if $\theta$ is replaced by $e^{\Upsilon}\theta$ for some smooth function $\Upsilon$ on $M$ then the functions $f_t$ must replaced by $e^{\Upsilon}f_t$. The family of functions $f_t$ may be more invariantly described as a family of densities (cf.\ \cite{CE2018-obstruction-flatI}), but this will not concern us since we will always work with respect to a fixed contact form. 
\end{remark}
\begin{proof}[Proof of Lemma \ref{lem:embed}]
Let $(T,Z_1^t,Z_{\oneb}^t)$ denote the frame dual to $(\theta, \theta^1_t,\theta^{\oneb}_t)$.  Since $\psi_t : M\to M_t$ is a diffeomorphism pulling back the CR structure on $M_t\subset X$ to the CR structure $(M,H,J_t)$ on $M$, the frame $(T, Z^t_1, Z^t_{\oneb})$ pushes forward under $\psi_t$ to a frame $(T^t, L^t_1, L^t_{\oneb})$ for $M_t$ such that $L^t_{\oneb}$ is a section of $T^{(0,1)}X|_{M_t}$ and $J T^t$ is everywhere transverse to $M_t$. The family of functions $f_t$ arises from taking components of $\dot{\psi}_t$ that are adapted to this framing as follows:  Since $\dot{\psi}_t\circ \psi_t^{-1}$ is a section of $TX|_{M_t}$ it can be written as
\beq \label{eqn:psi-t-dot-circ-psi-t-inv}
\dot{\psi}_t\circ \psi_t^{-1} = a_t JT^t + b_t T^t + c^1_t L^t_1 + c^{\oneb}_t L^t_{\oneb} 
\eeq 
for some (unique) smooth functions $a_t$, $b_t$, $c^1_t$ on $M_t$ with $a_t$ and $b_t$ real and $c^{\oneb}_t = \overline{c^1_t}$. Setting $f_t = (a_t+ib_t) \circ \psi_t$ and $V^{1}_t = -c^{1}_t \circ \psi_t$, equation \eqref{eqn:psi-t-dot-circ-psi-t-inv} becomes
\beq \label{eqn:psi-dot-components}
\dot{\psi}_t = {J\psi_t}_*\re\left(f_t T\right) -{\psi_t}_*\re\left(if_t T
+2V^{\oneb}_t Z^t_{\oneb}\right),
\eeq
where $V^{\oneb}_t$ is the conjugate of $V^1_t$. The smooth functions $f_t$ and $V^{1}_t$ on $M$ are uniquely determined for each $t$ and smooth in $t$, since $\psi_t$ is smooth in $t$. 

To obtain \eqref{eqn:psi-dot-from-ft} we need to show that $V^{\oneb}_t = \nabla^{\oneb}_t f_t$. In order to prove this and \eqref{Nabla2ft} it will be convenient to introduce (arbitrary) local coordinates $(z^1,z^2)$ on $X$. We write $\psi_t$ in these coordinates as $\psi_t=(\psi^1_t,\psi^2_t)$. Since, by definition, $\psi_t:(M,H,J_t)\to X$ is a CR embedding for each $t\in [0,\epsilon)$, it follows that the component functions $\psi^1_t$ and $\psi^2_t$ are CR functions with respect to $(M,H,J_t)$. Thus for all $t$ we have $Z_{\oneb}^t \psi_t^k =0$, $k=1,2$. Differentiating this expression with respect to $t$ we have $Z_{\oneb}^t \dot{\psi}_t^k + \dot{Z}_{\oneb}^t \psi_t^k=0$, $k=1,2$, where 
\begin{align*}
\dot{Z}_{\oneb}^t = \frac{d}{dt}Z_{\oneb}^t
&= \frac{\vfd_{\oneb}{}^1(t)}{\sqrt{1-|\vf|^2}} Z_1 \mod  Z^t_{\oneb}\\ 
&= \frac{\vfd_{\oneb}{}^1(t)}{1-|\vf|^2} Z^t_1 \mod  Z^t_{\oneb}, 
\end{align*}
using \eqref{ttransinv}. Hence, 
\beq\label{eqn:d-dt-of-CR-condition}
Z_{\oneb}^t \dot{\psi}_t^k + \frac{\vfd_{\oneb}{}^1(t)}{1-|\vf|^2} Z^t_1\psi_t^k = 0, \qquad k=1,2.
\eeq
The relationship between $V^1_t$ and $f_t$ will follow by substituting \eqref{eqn:psi-dot-components} into \eqref{eqn:d-dt-of-CR-condition}; to do this we first write \eqref{eqn:psi-dot-components} in terms of $(z^1,z^2)$-components by applying $dz^k$ to each side, giving
\begin{align}
\dot{\psi}_t^k &= i \left[{\psi_t}_*\re\left(f_t T\right) \right] z^k - \left[{\psi_t}_*\re\left( if_t T
+2V^{\oneb}_t Z^t_{\oneb}\right) \right] z^k \\
\nonumber &= i \re\left(f_t  T \right) \psi_t^k - \re\left( if_t T
+2V^{\oneb}_t Z^t_{\oneb}\right) \psi_t^k \\
\nonumber &= i \bar{f_t}  T  \psi_t^k 
-V^1_t Z^t_1 \psi_t^k, 
\end{align}
for $k=1,2$, where we have used that $d z^k \circ J = id z^k$ and that $Z^t_{\oneb}\psi^k_t=0$. Thus 
\beq\label{eqn:Z1bar-on-psik}
Z_{\oneb}^t \dot{\psi}_t^k = i(Z^t_{\oneb}\bar{f_t})T\psi_t^k + i \bar{f_t} Z^t_{\oneb}T\psi_t^k - (Z^t_{\oneb}V^1_t)Z^t_1\psi_t^k - V^1_t Z^t_{\oneb}Z^t_1\psi^k,
\eeq
for $k=1,2$. We can simplify the above by noting that $Z^t_{\oneb}T\psi_t^k = [Z^t_{\oneb},T]\psi_t^k$ and $Z^t_{\oneb}Z^t_1\psi^k = [Z^t_{\oneb},Z^t_1]\psi^k$, since $Z^t_{\oneb}\psi^k_t=0$, and using the following easy consequences of the Tanaka-Webster structure equations (cf.\ \cite[page 418]{Lee1986}):
\beq
[Z^t_{\oneb},Z^t_1]= ih^t_{1\oneb}T + \omega_1{}^1{}_{\oneb}(t)Z_1^t - \omega_{\oneb}{}^{\oneb}{}_1(t)Z_{\oneb}^t
\eeq
(where we retain the Levi form component $h^t_{1\oneb}$ even though it is $1$ for all $t$) and
\beq
[Z^t_{\oneb},T] = A^1{}_{\oneb}(t)Z_1^t - \omega_{\oneb}{}^{\oneb}{}_0(t)Z_{\oneb}^t.
\eeq
We obtain 
\beq
Z_{\oneb}^t \dot{\psi}_t^k = (iZ^t_{\oneb}\bar{f_t} - ih^t_{1\oneb}V^1_t )T\psi_t^k + (i\bar{f_t}A^1{}_{\oneb}(t) -Z^t_{\oneb}V^1_t - \omega_1{}^1{}_{\oneb}(t)V^1_t)Z^t_1\psi_t^k, 
\eeq
for $k=1,2$. Substituting this into \eqref{eqn:d-dt-of-CR-condition} we obtain
\beq \label{eqn:ft-and-V1-conditions}
(iZ^t_{\oneb}\bar{f_t} - ih^t_{1\oneb}V^1_t )T\psi_t^k + (i\bar{f_t}A^1{}_{\oneb}(t) -Z^t_{\oneb}V^1_t - \omega_1{}^1{}_{\oneb}(t)V^1_t + \frac{\vfd_{\oneb}{}^1(t)}{1-|\vf|^2} )Z^t_1\psi_t^k  =0,
\eeq
for $k=1,2$. Since $\psi_t:(M,H,J_t)\to X$ is a CR embedding for each $t$, we have
\begin{equation*}
\det \begin{pmatrix} Z^t_1\psi_t^1 &Z^t_1 \psi_t^2\\ T\psi^1_t & T\psi^2_t
\end{pmatrix}\neq 0
\end{equation*}
for all $t$. Hence \eqref{eqn:ft-and-V1-conditions} implies that $iZ^t_{\oneb}\bar{f_t} - ih^t_{1\oneb}V^1_t=0$ (i.e.\ $V^1_t = \nabla^1_t \bar{f_t}$, or equivalently, $V^{\oneb}_t = \nabla^{\oneb}_t f_t$) and
\beq \label{eqn:tangency-condition-messy}
i\bar{f_t}A^1{}_{\oneb}(t) -Z^t_{\oneb}V^1_t - \omega_1{}^1{}_{\oneb}(t)V^1_t + \frac{\vfd_{\oneb}{}^1(t)}{1-|\vf|^2} =0;
\eeq
recalling that the choice of local coordinate system $(z^1,z^2)$ was arbitrary, these identities hold on all of $M$ and for all $t$. Finally, using that $Z^t_{\oneb}V^1_t + \omega_1{}^1{}_{\oneb}(t)V^1_t = \nabla_{\oneb}^tV^1$ and that $V^1_t = \nabla^1_t \bar{f_t}$, and lowering an index using $h^t_{1\oneb}$ ($=1$ for all $t$), \eqref{eqn:tangency-condition-messy} gives
\beq
\nabla_{\oneb}^t\nabla_{\oneb}^t \bar{f_t} - iA_{\oneb\oneb}(t)\bar{f_t} = \frac{\vfd_{\oneb\oneb}(t)}{1-|\vf|^2},
\eeq
which is the conjugate of \eqref{Nabla2ft}. This proves the result.
\end{proof}

Note that the contact parametrization $\psi_t:M\to X$ in Lemma \ref{lem:embed} can always be replaced by one for which the corresponding functions $f_t$ are real-valued for all $t$ (the imaginary parts of the original functions $f_t$ generate a family of contact diffeomorphisms of $(M,H)$ that can be composed with the original $\psi_t$ to give the desired reparametrization, cf.\ \cite[Lemma 4.6]{CE2018-obstruction-flatI}). For our purposes, however, it will be advantageous to retain the flexibility of allowing $f_t$ to be complex.

Note also, that if $M_t$, $t\in [0,\epsilon]$, is a family of compact strictly pseudoconvex hypersurfaces in $\mathbb{C}^2$ then one can always find a smooth family of dilations $F_t$ such that the equivalent family $F_t(M_t)$ moves inwards for all $t$ and hence corresponds to a family of potentials with $\re f_t > 0$ for all $t$. Alternatively, by the same argument, $F_t(M_t)$ can be taken to move outwards for all $t$ corresponding to a family of potentials with $\re f_t < 0$. Thus, if we seek to characterize families of embeddable deformations via the existence of a family of solutions to \eqref{Nabla2ft}, it is no loss of generality in asking for our family $f_t$ to satisfy that $\re f_t$ has strict sign.

\section{Embeddings from Solutions to the Tangency Equation} \label{sec:embeddings-from-solutions}

In this section we shall show that the tangency condition does in fact characterize embeddability. More precisely, we shall show that it is possible to construct a smooth family of CR embeddings from a smooth family of potentials $f_t$ that solve the \emph{tangency equation} \eqref{Nabla2ft} for a given family of deformations $\vf(t)$; such a family of deformations $\vf(t)$ is said to satisfy the \emph{tangency condition}. The aim of this section is to prove the following result, from which Theorem \ref{thm:tangency-cond} directly follows.

\begin{theorem}\label{P:embed} Let $M=M_0$ be a compact smooth hypersurface bounding a strictly pseudoconvex domain $\Omega\subset\mathbb{C}^2$. Let $(M,H,J_t)$ be a smooth family of CR structures on $(M,H)$ with $J_0=J$ where $(M,H,J)$ is the CR structure induced on $M\subseteq \mathbb{C}^2$. Let $\vf(t)$ be the associated family of deformation tensors given by \eqref{Zdef}. Assume that there is a smooth family of solutions $f_t\in C^\infty(M,\bC)$ to the tangency equation \eqref{Nabla2ft} with $\re f_t$ having strict sign. 
Then, for a sufficiently small $\epsilon>0$, there is a family of mappings $\psi\colon M\times[0,\eps)\to \mathbb{C}^2$ such that:
\begin{itemize}
\item[(i)] $\psi_t\colon M\to \mathbb{C}^2$ is an embedding for each $t\in [0,\epsilon)$ with $\psi_0 =\mathrm{id}$, where $\psi_t:=\psi(\cdot,t)$.
\item[(ii)] $\psi_t$ is a CR diffeomorphism of $(M,H,J_t)$ onto the image $M_t:=\psi_t(M)\subset \mathbb{C}^2$.
\end{itemize}
\end{theorem}

The basic idea behind the proof of Theorem \ref{P:embed} is straightforward. We seek to solve the initial value problem
\begin{equation}\label{eqn:psi-dot-from-ft-IVP}
	\begin{dcases}
	\dot{\psi}_t = {J\psi_t}_*\re\left(f_t T\right) - {\psi_t}_*\re\left( if_t T
+2(\nabla^{\oneb}_t f_t) Z^t_{\oneb}\right), \\
	\psi_0 = \mathrm{id}.
	\end{dcases}
\end{equation}
If $(z^1,z^2)$ are the standard coordinates on $\mathbb{C}^2$ then $\psi_t$ can be written as $(\psi_t^1,\psi_t^2)$ and the initial value problem above becomes
\begin{equation}\label{eqn-psi-dot-local-coords}
	\begin{dcases}
	\left[\frac{\partial}{\partial t} -i\bar{f}_t T +  2\re\left((\nabla^{1}_t \bar{f}_t) Z^t_{1}\right) \right] \psi^j_t = 0,\\
	\psi^j_0 = z^j|_{M},
	\end{dcases}
\end{equation}
for $j=1,2$ (cf.\ the proof of Lemma \ref{lem:embed}). Since we expect $\psi_t$ to be a CR embedding with respect to the CR structure corresponding to $Z_1^t$ for each $t$, there is no loss of generality in supplementing \eqref{eqn-psi-dot-local-coords} with the additional equation $Z_{\oneb}^t\psi^j=0$, $j=1,2$, for each $t$. With this in mind, the local existence of a solution to \eqref{eqn-psi-dot-local-coords} in the case where $\re f_t > 0$ follows easily from the Newlander-Nirenberg theorem with boundary and extension of CR functions (applied to $z^j|_{M}$, $j=1,2$). Global existence then follows by patching the local solutions using local uniqueness. In the case where $\re f_t <0$ we are able to use the classical version of the Newlander-Nirenberg theorem, but we can no longer use extension of CR functions to conclude. We get around this by using the (stable) embeddability results of Lempert and Epstein-Henkin to obtain a solution $\tilde{\psi}_t$ of \eqref{eqn:psi-dot-from-ft-IVP} except with the initial condition possibly being some CR embedding other than the identity. It is then easy to modify this family so that the initial condition is satisfied; we remark that in doing so we are effectively (indirectly) solving \eqref{eqn:psi-dot-from-ft-IVP} with respect to a new family of potentials $f_t$ solving the tangency equation for the same family of deformations $\vf(t)$ with $\re f_t>0$.

\begin{proof}[Proof of Theorem \ref{P:embed}]
Let $\vf(t)$, $f_t$, $t\in [0,\epsilon_0)$, be given as in the statement of the theorem. Let $\overline{W}:=\frac{\partial}{\partial t} -i\bar{f}_t T +  2\re\left((\nabla^{1}_t \bar{f}_t) Z^t_{1}\right)$. On $Y=M\times[0,\epsilon_0)$ we consider the almost complex structure given by 
\beq
T^{0,1}Y = \mathrm{span}\{Z^t_{\oneb}, \overline{W}\}.
\eeq
A straightforward calculation using the tangency equation \eqref{Nabla2ft} shows that $T^{0,1}Y$ is in fact integrable (cf.\ also the proof of Theorem A in \cite{Cheng2014} where this calculation is done for the case corresponding to $\re f_t <0$; note that the calculation is independent of the sign of $\re f_t$). Note that this complex structure induces the CR structure corresponding to $Z_{\oneb}^t$ on each slice $M\times\{t\}$. In the following we identify $M\subset \mathbb{C}^2$ with $M\times\{0\}$, though we retain both notations to help indicate when we are working in $\mathbb{C}^2$ and when we are working in $Y$. 

We now show that in the case when $\re f_t >0$ the complex manifold $Y=M\times[0,\epsilon_0)$ with (partial) boundary $M\times\{0\}$ can be realized as a one sided neighborhood of $M\subset \mathbb{C}^2$ with $M$ as the strictly pseudoconvex boundary since $\re f_t>0$. 
Indeed, by the Newlander-Nirenberg theorem with boundary \cite{Catlin1988,HangesJacobowitz1989} the complex manifold $Y=M\times[0,\epsilon_0)$ can be locally realized in a neighborhood of any point $p\in M\times \{0\}$ as the strictly pseudoconvex side of a strictly pseudoconvex hypersurface in $\mathbb{C}^2$ (it remains to be shown that the realization can be taken to be the identity on $M_0\subset \mathbb{C}^2$). 
From this local realizability it follows that a CR function on $M\times \{0\}$ can be extended to a holomorphic function in a neighborhood in $Y=M\times[0,\epsilon_0)$ of any point $p\in M\times\{0\}$. 
In particular, the CR functions $\psi^1_0=z^1|_M$ and $\psi^2_0=z^2|_M$ admit unique such local extensions. 
Moreover, by the local uniqueness these extensions must glue together to unique extensions $\psi^1(x,t) = \psi^1_t(x)$ and $\psi^2(x,t) = \psi^2_t(x)$ of $\psi^1_0(x)$ and $\psi^2_0(x)$, defined in a neighborhood $M\times [0,\epsilon)$ of $M\times\{0\}$, solving $\overline{W}\psi^j=0$ and $Z_{\oneb}^t\psi^j=0$ for $j=1,2$. The family of maps $\psi_t=(\psi^1_t,\psi^2_t):M\to \mathbb{C}^2$, $t\in [0,\epsilon)$ therefore solves \eqref{eqn:psi-dot-from-ft-IVP}.

It remains to consider the case where $\re f_t <0$ for all $t$. As a first step, we observe that \eqref{eqn:psi-dot-from-ft-IVP} is formally solvable to all orders at $t=0$, and hence (by Borel's lemma) there exists a smooth family of contact embeddings $\hat{\psi}_t:M\to X$, $t\in (-\delta,0]$ satisfying \eqref{eqn:psi-dot-from-ft-IVP} to all orders (for our given $f_t$, $\vf(t)$) at $t=0$. By Lemma \ref{lem:embed} there exist a family of deformation tensors $\hat{\vf}(t)$ and potentials $\hat{f}_t$, $t\in (-\delta,0]$ corresponding to the family of embeddings $\hat{\psi}_t$ and by construction $\hat{\vf}(t)$ and $\hat{f}_t$ agree with $\vf(t)$ and $f_t$, respectively, to all orders at $t=0$; taking $\delta>0$ sufficiently small we can ensure that $\re \hat{f}_t<0$ for all $t$. 
Hence we may smoothly extend the families $\vf(t)$ and $f_t$ from $t\in [0,\epsilon_0)$ to $t\in (-\delta,\epsilon_0)$ by defining $\vf(t)$ and $f_t$ to be equal to $\hat{\vf}(t)$ and $\hat{f}_t$, respectively, when $t\in (-\delta,0)$. 
Correspondingly we extend $Z^t_1$ and $W$ from $t\in [0,\epsilon_0)$ to $t\in (-\delta,\epsilon_0)$ in the obvious way. 
The almost complex structure on $Y$ therefore extends to an almost complex structure on $\tilde{Y} = M\times(-\delta,\epsilon_0)$ with $T^{0,1}\tilde{Y}=\mathrm{span}\{Z^t_{\oneb}, \overline{W}\}$. 
By construction, since $\hat{\vf}(t)$ and $\hat{f}_t$ also solve the tangency equation \eqref{Nabla2ft} for $t\in (-\delta,0]$, the almost complex structure on $\tilde{Y}$ is integrable, making $\tilde{Y}$ a complex manifold.
Moreover, by construction, the map 
\beq\label{one-sided-biholom}
\hat{Y}=M\times(-\delta,0]\ni(x,t)\to \hat{\psi}_t(x) \in \mathbb{C}^2
\eeq
is a diffeomorphism from $\hat{Y}$ onto a one sided neigborhood $\mathcal{U}$ of $M\subset X$ containing $M$ as its strictly pseudoconvex (partial) boundary that restricts to a biholomorphism from $M\times(-\delta,0)$ to $\mathcal{U}\setminus M$. We can therefore extend the domain $\Omega\subset \mathbb{C}^2$ bounded by $M$ to a complex manifold $X$ containing $\tilde{Y}$ by identifying $M\times(-\delta,0)\subset \tilde{Y}$ with $\mathcal{U}\setminus M\subset \Omega$ via the map \eqref{one-sided-biholom}.
 (Note that the result of this paragraph was essentially also established in the proof of Theorem A in \cite{Cheng2014}, though we find the approach we have taken here to be simpler.) 

From the construction of $X$ and the fact that $\Omega$ is strictly pseudoconvex it is clear that $\Omega$ has a Stein neighborhood basis in $X$. Moreover, if $X_t = \Omega \cup \left(M\times[0,t)\right) \subset X$ then $X_t$ is a Stein space for all sufficiently small $t>0$. In particular, there exists $\epsilon>0$ such that $M_t=\partial X_t = M\times\{t\}$ is Stein fillable and hence embeddable in $\mathbb{C}^2$ for $t\in [0,\epsilon]$ (see \cite{Lempert1994}, cf.\ \cite{EpsteinHenkin1997,EpsteinHenkin2000}). It follows that $X_{\epsilon}$ can be holomorphically embedded in $\mathbb{C}^2$ (by an embedding that is smooth up the boundary; again, see \cite{Lempert1994,EpsteinHenkin1997, EpsteinHenkin2000}). Let $\tilde{\psi}:X_{\epsilon}\to \mathbb{C}^2$ denote such an embedding. Clearly $\tilde{\psi}_t=\psi|_{M_t}$ is a CR embedding realizing the CR structure corresponding to $Z_1^t$ on $M_t$ for each $t$. However, $\tilde\psi_0:M=M_0\to \mathbb{C}^2$ may not be the identity; i.e.\ the family $\tilde\psi_t$ solves \eqref{eqn:psi-dot-from-ft-IVP} except with a different CR embedding as the initial condition. Let $\tilde{\Omega}_t=\tilde{\psi}(X_t)$ and $\tilde{M}_t=\tilde{\psi}_t(M_t)$. By construction $\tilde\psi|_{\Omega}$ is a biholomorphism from $\Omega_0=\Omega$ to $\tilde{\Omega}_0$ that is smooth up to the boundary. To conclude the proof we show that the map $\tilde\psi|_{\Omega}^{-1}: \tilde{\Omega}_0\to \Omega_0\subset\mathbb{C}^2$ can be extended to a smooth (in the appropriate sense) family of maps $\Phi_t:\tilde{\Omega}_t\to \mathbb{C}^2$, $t\in [0,\epsilon]$, such that $\Phi_t$ is a biholomorphism onto its image $\Omega_t$ that is smooth up the boundary for each $t$. One way to accomplish this is to note that, by compactness of the $\tilde{M}_t$ and of $[0,\epsilon]$, there exists a point in $\tilde\Omega_0$ (which, without loss of generality, we take to be the origin) and a constant $\alpha>0$ large enough such that the smooth family of dilation maps $F_t:\mathbb{C}^2\ni\tilde{z}\to (1+\alpha t)\tilde{z}\in\mathbb{C}^2$ maps $\tilde\Omega_0$ to a domain $F_t(\tilde\Omega_0)$ containing $\tilde{\Omega}_t$ for each $t\in [0,\epsilon]$. The desired maps $\Phi_t$ can then be taken to be $\Phi_t= \tilde\psi|_{\Omega}^{-1} \circ F_t^{-1}|_{\tilde{\Omega}_t}$. Then $\psi_t = \Phi_t \circ \tilde\psi_t = \tilde\psi|_{\Omega}^{-1} \circ F_t^{-1} \circ \tilde\psi_t$ gives the desired family of embeddings realizing the CR structure on $M_t$ for each $t$. This concludes the proof.
\end{proof}

\begin{remark}\label{rem:analyticity-of-psi-t}
If $\vf(t)$ and $f_t$ in Theorem \ref{P:embed} are analytic in $t$ (with values in some Banach space of functions), then one can show that $\psi_t$ and $\gamma_t$ are also analytic in $t$ (with values in the corresponding Banach spaces).
\end{remark}
Note that in the case $\re f_t >0$ it is easy to see that Theorem \ref{P:embed} can be strengthened by replacing $\mathbb{C}^2$ with any complex surface.

\section{Solutions to the Tangency Equation on $S^3$}\label{sec:solutions-tangency-eqn}

We are now going to consider deformations of the standard CR structure on $S^3$. Recall that a smooth infinitesimal deformation tensor $\dot{\varphi}_{11}$ on $S^3$ is embeddable, i.e.\ there is a family of embeddable deformations $\varphi_{11}(t)$ of the unit sphere $S^3$ in $\mathbb{C}^2$ with $\vf_{11}(0)=0$ such that $\dot{\varphi}_{11}=\left.\frac{d}{dt}\right|_{t=0}\vf_{11}(t)$, if and only if
$$
(Z_1)^2f = \dot{\varphi}_{11}
$$
for some $f\in C^{\infty}(S^3, \mathbb{C})$, where $Z_1$ is given by \eqref{eqn:S3-Z1}. In this section we shall construct, in a canonical way, such a deformation $\varphi_{11}(t)$ for a given embeddable infinitesimal deformation tensor $\dot{\varphi}_{11}$. We do this by constructing a smooth family of  complex functions $f_t$ and a smooth family of deformations $\varphi_{11}(t)$ satisfying \eqref{Nabla2ft} and then appealing to Theorem \ref{P:embed}.  We start with some preliminary calculations.

\subsection{Deformations of $S^3$ with its standard CR structure}

We shall continue to use the notation of Section \ref{sec:deformations}, where we now take $(M,H,J,\theta)$ to be the standard pseudohermitian structure on $S^3$. Recall that $\theta = i(z d\bar{z} +wd\bar{w})$ restricted to the unit sphere $S^3$, where $(z,w)$ are the coordinates on $\mathbb{C}^2$. We also recall that we are using the frame $Z_1$ as in \eqref{eqn:S3-Z1}. We observe that in this case we have
\beq\label{eqn:pseudohermitian-connection}
\omega_1{}^1=-iR\theta, \quad A_{11}=0
\eeq
where $R=2$ is the Tanaka-Webster scalar curvature of $S^3$ with its standard structure.


Let $\vf_1{}^{\bar 1}(t)$ be a smooth family of deformation tensors on $S^3$ with its standard pseudohermitian structure and frame. We shall consider the family of pseudohermitian structures defined by the admissible coframes $(\theta, \theta^1_t, \theta^{\oneb}_t)$ where $\theta^1_t$ is defined as in \eqref{th1def}.
We shall compute $\omega_1{}^1(t)$ and $A_{11}(t)$ in terms of the deformation tensor $\vf_1{}^{\bar 1}(t)$.  We differentiate $\theta^1_t$ and then substitute for $\theta^1$ and $\theta^{\bar 1}$ using \eqref{ttransinv}:
\beq\label{dtheta1t}
\begin{aligned}
d\theta^1_t &= \frac{1}{2}\frac{d|\vf|^2}{(1-|\vf|^2)^{3/2}}\wedge(\theta^1-\vf_{\bar 1}{}^1\theta^{\bar 1})+\frac{1}{\sqrt{1-|\vf|^2}}(d\theta^1-d\vf_{\bar 1}{}^1\wedge\theta^{\bar 1}-\vf_{\bar 1}{}^1d\theta^{\bar 1})\\
&=-\frac{1}{2}\theta^1_t\wedge \frac{d|\vf|^2}{1-|\vf|^2}+\frac{1}{\sqrt{1-|\vf|^2}} (-iR\theta^1\wedge\theta-iR\vf_{\bar 1}{}^1 \theta^{\bar 1}\wedge\theta+\theta^{\bar 1}\wedge d\vf_{\bar 1}{}^1)\\
&=\frac{1}{1-|\vf|^2}\theta^1_t\wedge(-\frac{d|\vf|^2}{2}-iR\theta-iR|\vf|^2\theta +\vf_1{}^{\bar 1}d\vf_{\bar 1}{}^1-(Z^t_1\lrcorner d\vf_{\bar 1}{}^1)\theta^{\bar 1}_t)-\frac{\vf_{\bar 1}{}^1{}_{,0}}{1-|\vf|^2}\theta\wedge\theta^{\bar 1}_t
\end{aligned}
\eeq
where $\vf_1{}^{\bar 1}=\vf_1{}^{\bar 1}(t)$ and we have used that $T\lrcorner d\vf_{\bar 1}{}^1-2iR\vf_{\bar 1}{}^1 = \vf_{\bar 1}{}^1{}_{,0}$. We note that
$$
d|\vf|^2=\vf_1{}^{\bar 1}d\vf_{\bar 1}{}^1+\vf_{\bar 1}{}^1d\vf_1{}^{\bar 1},
$$
and hence
\beq \label{connmodtheta}
\begin{aligned}
\frac{d|\vf|^2}{2}-\vf_1{}^{\bar 1}d\vf_{\bar 1}{}^1 =&\,\frac{1}{2}\,(-\vf_1{}^{\bar 1}d\vf_{\bar 1}{}^1+\vf_{\bar 1}{}^1d\vf_1{}^{\bar 1})\\
=&\,\frac{1}{2}\left(-\vf_1{}^{\bar 1}(\vf_{\bar 1}{}^1{}_{,1}\theta^1+\vf_{\bar 1}{}^1{}_{,\bar 1}\theta^{\bar 1})+\vf_{\bar 1}{}^{1}(\vf_{1}{}^{\bar  1}{}_{,1}\theta^1+\vf_1{}^{\bar 1}{}_{,\bar 1}\theta^{\bar 1})\right)\mod\theta\\
=&\,\frac{1}{2\sqrt{1-|\vf|^2}}\,(-|\vf|^2\vf_{\bar 1}{}^1{}_{,1}-\vf_1{}^{\bar 1}\vf_{\bar 1}{}^1{}_{,\bar 1}+(\vf_{\bar 1}{}^1)^2\vf_1{}^{\bar 1}{}_{,1}+\vf_{\bar 1}{}^1\vf_1{}^{\bar 1}{}_{,\bar 1})\theta^{\bar 1}_t\,
\end{aligned}
\eeq
where the last equality is modulo $\theta$ and $\theta^1_t$. We also note that the $\theta$-component is given by
\begin{equation}\label{conntheta}
\frac{d|\vf|^2}{2}-\vf_1{}^{\bar 1}d\vf_{\bar 1}{}^1 = \left(\frac{1}{2}(-\vf_1{}^{\bar 1}\vf_{\bar 1}{}^1{}_{,0}+\vf_{\bar 1}{}^{1}\vf_{1}{}^{\bar  1}{}_{,0})-2iR|\vf|^2\right) \theta\mod \theta_t^1,\, \theta_t^{\bar 1}.
\end{equation}
Next, we note that
\beq\label{Zidvf}
Z^t_1\lrcorner d\vf_{\bar 1}{}^1=\frac{1}{\sqrt{1-|\vf|^2}}\,(\vf_{\bar 1}{}^1{}_{,1}+\vf_{1}{}^{\bar 1}\vf_{\bar 1}{}^1{}_{,\bar 1}).
\eeq
By using also the fact that $\omega_1{}^1(t)$ is purely imaginary (since $h^t_{11}=1$), the following lemma follows from the calculations \eqref{dtheta1t}--\eqref{Zidvf} and the structure equation \eqref{eqn:pseudohermitian-connection1}:

\begin{lemma}\label{Lem:defpsh} Let $M=S^3$ with its standard pseudohermitian structure. Then,
\beq
\begin{aligned}
\omega_1{}^1{}_1(t)&=\frac{1}{2(1-|\vf|^2)^{3/2}}\,(2\vf_{1}{}^{\bar 1}{}_{,\bar 1}+\vf_{1}{}^{\bar 1}\vf_{\bar 1}{}^{ 1}{}_{, 1}+\vf_{\bar 1}{}^1\vf_1{}^{\bar 1}{}_{,1}+(\vf_{1}{}^{\bar 1})^2\vf_{\bar 1}{}^{1}{}_{,\bar 1}-|\vf|^2\vf_{1}{}^{\bar 1}{}_{,\bar 1})\\
\omega_1{}^1{}_0(t)&=-2i+\frac{1}{2(1-|\vf|^2)}\,(\vf_1{}^{\bar 1}\vf_{\bar 1}{}^1{}_{,0}-\vf_{\bar 1}{}^{1}\vf_{1}{}^{\bar  1}{}_{,0})
\\
A_{11}(t)&=-\frac{\vf_{11,0}}{1-|\vf|^2}.
\end{aligned}
\eeq

\end{lemma}

\begin{remark} The expression for $\omega_1{}^1{}_1(t)$ can be simplified by noting that
\begin{align*}
(Z_1+\vf_1{}^{\oneb}Z_{\oneb})\,\frac{1}{\sqrt{1-|\vf|^2}}&=\frac{1}{2(1-|\vf|^2)^{3/2}}(Z_1+\vf_1{}^{\oneb}
Z_{\oneb})|\vf|^2\\&=
\frac{1}{2(1-|\vf|^2)^{3/2}}(\vf_{1}{}^{\bar 1}\vf_{\bar 1}{}^{ 1}{}_{, 1}+\vf_{\bar 1}{}^1\vf_1{}^{\bar 1}{}_{,1}+(\vf_{1}{}^{\bar 1})^2\vf_{\bar 1}{}^{1}{}_{,\bar 1}+|\vf|^2\vf_{1}{}^{\bar 1}{}_{,\bar 1}).
\end{align*}
Consequently, we may write
\beq\label{eq:newconn}
\omega_1{}^1{}_1(t)=\frac{\vf_{1}{}^{\bar 1}{}_{,\bar 1}}{\sqrt{1-|\vf|^2}}+\tilde Z^t_1\,\frac{1}{\sqrt{1-|\vf|^2}},
\eeq
where
\beq\label{eq:tildeZ}
\tilde Z^t_1=Z_1+\vf_1{}^{\oneb}Z_{\oneb}.
\eeq
\end{remark}

\subsection{Spherical harmonics}

We shall denote the space of spherical harmonic polynomials of bidegree $(p,q)$ on $S^3\subset \bC^2$ by $H_{p,q}$. We recall that the spherical harmonic spaces $H_{p,q}$ are eigenspaces for $T = i\left(z\frac{\partial}{\partial z} + w\frac{\partial}{\partial w}\right)- i\left(\bar{z}\frac{\partial}{\partial \bar{z}} + \bar{w}\frac{\partial}{\partial \bar{w}}\right)$ acting on functions,
\beq \label{eqn:Reeb-spherical-harmonics}
Tu =i(p-q)u,
\eeq
and that $Z_1$ maps $H_{r,s}$ isomorphically onto $H_{r-1,s+1}$ when $r\geq 1$ and $Z_1=0$ on $H_{0,s}$. An immediate consequence of this is that $\dot{\varphi}_{11}$ is in the image of $(Z_1)^2$ if and only if the spherical harmonic expansion of $\dot{\varphi}_{11}$ has vanishing components in $H_{p,q}$ for $q=0,1$, and the kernel of $(Z_1)^2$ is given by those complex functions whose only nontrivial components are in $H_{p,q}$ for $p=0,1$. It follows that if $\vfd_{11}$ is in the image of $(Z_1)^2$ then there is a unique complex function $f$ such that $(Z_1)^2f = \vfd_{11}$ and $f$ has vanishing components in $H_{p,q}$ for $p=0,1$, i.e.\ $f$ is $L^2$-orthogonal to the kernel of $(Z_1)^2$. Note that $\nabla_1$ always acts as $Z_1$ on tensors of any type, since the connection form is given by $\omega_1{}^1=-iR\theta$. 

From the above it follows that the sublaplacian $\Delta_b = -Z_1 Z_{\oneb} - Z_{\oneb}Z_1$ maps each $H_{p,q}$ to itself. Since $\Delta_b$ is $\mathrm{SU}(2)$-invariant (indeed $Z_1$ itself is) it must act on each $H_{p,q}$ by a constant (by Schur's lemma, since the $H_{p,q}$ are irreducible representations of $\mathrm{SU}(2)$). The constant is easily seen to be $2pq+p+q$ (e.g., apply $\Delta_b$ to the spherical harmonic $(z+w)^p(\bar{z}-\bar{w})^q$). Note that the Folland-Stein Sobolev $s$-norm $\norm u \norm_s$ on $H_{FS}^s$ is equivalent to the norm
\beq\label{eq:FS-norm}
 \norm (1+ \Delta_b)^{s/2} u\norm_{L^2} = \left(\sum_{p,q} (1+p+q+2pq)^{s} \norm u_{p,q}\norm_{L^2}^2 \right)^{1/2}
\eeq
where $u= \sum_{p,q} u_{p,q}$ \cite{FollandStein1974,JerisonLee1987}. To normalize constants it will be convenient to take \eqref{eq:FS-norm} as the definition of the $H^s_{FS}$-norm in the following.

\subsection{Formally embeddable deformations}

We now assume that $\vf_{11}(0)=0$. We shall expand a potential solution $f_t$ and deformation tensor $\varphi_{11}(t)$ satisfying \eqref{Nabla2ft} in powers of $t$ as follows:
\begin{align}
\label{e:ft}
f_t &=\sum_{k=0}^\infty f^{(k)}t^k\\
\label{e:phit}
\varphi_{11}(t) &=\sum_{k=1}^{\infty} \varphi^{(k)}t^k,
\end{align}
where $\vf^{(1)}{}=\dot\vf_{11}$. We shall identify terms in \eqref{Nabla2ft} with equal powers of $t$. We obtain for $t^0$:
\beq\label{e:k=0}
(\nabla_1)^2 f^{(0)}=\dot\vf_{11},
\eeq
the solvability of which is equivalent to $\dot\vf_{11}$ being an embeddable infinitesimal deformation. Before we proceed we first rewrite equation \eqref{Nabla2ft} explicitly in terms of the deformation $\vf_{11}(t)$.  We note that $\omega_1{}^1(0)=-Ri\theta$ implies that $\nabla_1=\nabla_1^t|_{t=0}$ and $\nabla_{\oneb}=\nabla_{\oneb}^t|_{t=0}$ act on any tensor simply as $Z_1$ and $Z_{\oneb}$. The left hand side of \eqref{Nabla2ft} can be written, by using the expression for $A_{11}(t)$ in Lemma \ref{Lem:defpsh},
\begin{multline}
\frac{1}{1-|\vf|^2}\left(\nabla_1+\vf_{1}{}^{\oneb }\nabla_{\oneb}\right)^2f_t
-\omega_1{}^1{}_1\frac{1}{\sqrt{1-|\vf|^2}}\left(\nabla_1+\vf_{1}{}^{\oneb }\nabla_{\oneb}\right)f_t\\+\frac{1}{\sqrt{1-|\vf|^2}}\left(
\left(\nabla_1+\vf_{1}{}^{\oneb}\nabla_{\oneb}\right)\frac{1}{\sqrt{1-|\vf|^2}}\right)
\left(\nabla_1+\vf_{1}{}^{\oneb }\nabla_{\oneb}\right)f_t -\frac{i\vf_{11,0}}{1-|\vf|^2}\,f_t,
\end{multline}
where we have also abbreviated $\omega_1{}^1{}_1=\omega_1{}^1{}_1(t)$ and $\vf_{11}=\vf_{11}(t)$. By using also \eqref{eq:newconn}, we find that this simplifies to
\begin{equation}
\frac{1}{1-|\vf|^2}\left(\left(\nabla_1+\vf_{1}{}^{\oneb }\nabla_{\oneb}\right)^2f_t - \vf_{1}{}^{\bar 1}{}_{,\bar 1}\left(\nabla_1+\vf_{1}{}^{\oneb }\nabla_{\oneb}\right)f_t
 -i\vf_{11,0}f_t\right).
\end{equation}
By canceling a factor of $(1-|\vf|^2)^{-1}$ in \eqref{Nabla2ft}, we obtain the equation
\begin{equation}\label{eq:embeddability}
\left(\nabla_1+\vf_{1}{}^{\oneb }\nabla_{\oneb}\right)^2f_t - \vf_{1}{}^{\bar 1}{}_{,\bar 1}\left(\nabla_1+\vf_{1}{}^{\oneb }\nabla_{\oneb}\right)f_t
 -i\vf_{11,0}f_t= \frac{d}{dt} \vf_{11}.
\end{equation}
The operator acting on $f_t$ on the left hand side of this equation can be expressed as $(\nabla_1)^2 + L_\vf$ where
\begin{equation}\label{eq:L}
L_\vf=\vf_{1}{}^{\oneb}\nabla_1\nabla_{\oneb} +\vf_{1}{}^{\oneb}\nabla_{\oneb}\nabla_1+
(\vf_{1}{}^{\oneb})^2(\nabla_{\oneb})^2+\vf_{1}{}^{\oneb}{}_{,1}\nabla_{\oneb}-
\vf_{1}{}^{\oneb}{}_{,\oneb}\nabla_1-i\vf_{11,0}.
\end{equation}
We note that $L_\vf$ has a Taylor expansion
\beq
L_\vf = \sum_{k=1}^{\infty} t^k L^{(k)}
\eeq
where the operators $L^{(k)}$ are given by
\begin{multline}
L^{(k)} =
\vf^{(k)}\nabla_1\nabla_{\oneb} +\vf^{(k)}\nabla_{\oneb}\nabla_1+
\sum_{j=1}^{k-1}\vf^{(j)}\vf^{(k-j)} (\nabla_{\oneb})^2\\+(\nabla_1\vf^{(k)})\nabla_{\oneb}-
(\nabla_{\oneb}\vf^{(k)})\nabla_1-i\nabla_0\vf^{(k)},
\end{multline}
where we have used the notation in \eqref{e:phit} and recall that $h_{1\oneb}=1$, so $\vf_{1}{}^{\oneb}=\vf_{11}$.

For the proof of the following proposition we introduce the orthogonal (in $L^2$) projections $\mathcal P_1,\mathcal P_2$ onto the image of $(\nabla_1)^2$ (i.e., the subspace of functions with vanishing components in $H_{p,q}$ for $q=0,1$) and its orthogonal complement, the kernel of $(\nabla_{\oneb})^2$ (i.e., with non-vanishing components only in $H_{p,q}$ for $q=0,1$).

\begin{proposition}\label{prop:formalSol}
Given a smooth embeddable infinitesimal deformation tensor $\dot{\varphi}_{11}$, there are unique formal power series $f_t =\sum_{k=0}^\infty f^{(k)}t^k$ and $\varphi_{11}(t) = t\vfd_{11}+\sum_{k=2}^{\infty} \varphi^{(k)}t^k$, with $f^{(k)}$ and $\vf^{(k)}$ smooth, satisfying \eqref{eq:embeddability} such that for each $k$, $f^{(k)}$ has vanishing components in $H_{p,q}$ for $p=0,1$, and for each $k\geq 2$, $\varphi^{(k)}$ has vanishing components in $H_{p,q}$ for $q \geq 2$.
\end{proposition}
\begin{remark} Note that $\vf(t)$ is of the form $t\vfd + \psi(t)$ where $\psi(t) = \sum_{k=2}^{\infty} \varphi^{(k)}t^k$ takes values in $\mathfrak{D}_0^{\perp}$.
\end{remark}

\begin{proof}
By identifying coefficients of $t^k$ in \eqref{eq:embeddability} we get for $t^0$, $(\nabla_1)^2 f^{(0)}=\dot\vf_{11}$, and for $t^k$, $k\geq 1$,
\beq\label{eq:fkphik+1}
(\nabla_1)^2f^{(k)} +  \sum_{j=1}^{k} L^{(j)}f^{(k-j)} = (k+1)\varphi^{(k+1)}.
\eeq
We take $f^{(0)}$ to be the unique solution of $(\nabla_1)^2 f^{(0)}=\dot\vf_{11}$ with vanishing components in $H_{p,q}$ for $p=0,1$. For $k\geq 1$ we define $f^{(k)}$ and $\varphi^{(k+1)}$ recursively by decomposing $\sum_{j=1}^{k} L^{(j)}f^{(k-j)}= A_k + B_k$, where
\beq
A_k=\mathcal P_1 \sum_{j=1}^{k} L^{(j)}f^{(k-j)},\quad B_k=\mathcal P_2 \sum_{j=1}^{k} L^{(j)}f^{(k-j)},
\eeq
and then defining $f^{(k)}$ to be the unique solution to $(\nabla_1)^2f^{(k)} = -A_k$ with vanishing components in $H_{p,q}$ for $p=0,1$, and $\varphi^{(k+1)}$ to be $B_k/(k+1)$. The solutions are easily seen to be smooth by standard properties of the solution operator to $(\nabla_1)^2$. This concludes the proof.
\end{proof}
\begin{remark}
Note that in Proposition \ref{prop:formalSol} we could have instead allowed the components of the $f^{(k)}$ in $H_{p,q}$ for $p=0,1$ to be arbitrary, since $(\nabla_1)^2$ annihilates $H_{p,q}$ for $p=0,1$. Doing this we obtain the general formal solution to the tangency equation \eqref{eq:embeddability}. Below we shall use this flexibility and allow $f^{(0)}$ to have a nontrivial component in $H_{0,0}$.
\end{remark}

\subsection{Deformations in the Burns-Epstein region}

Recall from the introduction that the space of Burns-Epstein deformations $\mathfrak{D}_{BE}$ is the set of all deformation tensors $\vf$ such that $\vf_{p,q}=0$ if $q<p+4$. The following lemma follows easily by inspection of the definition of $L_{\vf}$ given in \eqref{eq:L}.
\begin{lemma}\label{lem:BE-phi-f}
Let $\varphi_{11}\in \mathfrak{D}_{BE}$ and $f\in (Z_{\oneb})^2 \mathfrak{D}_{BE} $. Then $L_{\varphi}f \in \mathfrak{D}_{BE}$.
\end{lemma}
An immediate consequence of this is the following:
\begin{lemma}\label{lem:BE-induction}
Let $\varphi_{11}(t) =\sum_{k=1}^{\infty} \varphi^{(k)}t^k$ and $f_t =\sum_{k=0}^\infty f^{(k)}t^k$ be formal power series with values in $C^{\infty}(S^3,\mathbb{C})$. If $\varphi^{(j)} \in \mathfrak{D}_{BE}$, $1\leq j\leq k$, and $f^{(j)}\in (Z_{\oneb})^2 \mathfrak{D}_{BE}$, $0\leq j \leq k-1$, then $\sum_{j=1}^{k} L^{(j)}f^{(k-j)} \in \mathfrak{D}_{BE}$.
\end{lemma}
\begin{proof}
The lemma follows by applying Lemma \ref{lem:BE-phi-f} to $\tilde{\varphi}_{11} = \sum_{j=1}^{k} \varphi^{(j)}t^j$ and $\tilde{f}_t = \sum_{j=0}^{k-1} f^{(j)}t^j$ and then taking the $t^k$ coefficient of $L_{\tilde{\varphi}}\tilde{f}$, which is the same as the $t^k$ coefficient of $L_{\varphi}f$.
\end{proof}

We let $H_{FS}^s$ denote the Folland-Stein Sobolev space \cite{FollandStein1974} of complex valued functions on $S^3$ with $s$ derivatives in $L^2$ in the directions tangent to the contact distribution $H$. (Note that on these spaces any Reeb vector field, being a commutator of vector fields tangent to $H$, behaves like a second order operator.) We denote the norm on $H_{FS}^s$ by $\norm\cdot\norm_s$ (we will also occasionally use the standard Sobolev norm, which we denote by $\norm\cdot\norm_{H^s}$). Note that $\mathfrak{D}_{BE}$ is a subspace of the space $\mathfrak{D}_0$ of deformation tensors with vanishing component in $H_{p,q}$ for $q=0,1$. Thus Proposition \ref{prop:formalSol} can be applied to any infinitesimal deformation tensor in $\mathfrak{D}_{BE}$, and using Lemma \ref{lem:BE-induction} we will prove the following:

\begin{proposition} \label{prop:convergence-ft}
Given $\dot{\varphi}_{11}\in \mathfrak{D}_{BE}$, the unique formal power series $f_t =\sum_{k=0}^\infty f^{(k)}t^k$ and $\varphi_{11}(t) =\sum_{k=1}^{\infty} \varphi^{(k)}t^k$ given by Proposition \ref{prop:formalSol} satisfy
\begin{equation}\label{eqn:BE-tphidot}
\varphi_{11}(t) = t\dot{\varphi}_{11} \qquad \text{and} \qquad f^{(k)} \in (Z_{\oneb})^2 \mathfrak{D}_{BE}
\end{equation}
for all $k$. Moreover, for every $s\geq 10$ there is $C_s>0$ such that the formal power series $f_t =\sum_{k=0}^\infty f^{(k)}t^k$ converges for $|t|< R_s = (C_s\norm\dot{\varphi}_{11} \norm_s )^{-1}$ to an analytic function taking values in $H_{FS}^s$, and for each fixed $t$, $|t|< R_s$, $f_t$ is a $C^{\infty}$ function on $S^3$.
\end{proposition}
\begin{remark}
In this paper we will not be concerned with optimal regularity in the finite regularity case. The choice $s\geq 10$ in the proposition is only for convenience and is not optimal.
\end{remark}

In the proof we will make use of the standard solution operator for the Kohn Laplacian $\square_b = -\nabla^{\oneb}\nabla_{\oneb}$ on $S^3$, denoted as in \cite{BurnsEpstein1990b} by $\mathcal{Q}_0$. That is, $\mathcal{Q}_0$ acts by zero on $\mathrm{ker}\,\square_b=\mathrm{ker}\,Z_{\oneb}$ and is the inverse to $\square_b$ on $(\mathrm{ker}\,Z_{\oneb})^{\perp}$. Note that $\square_b=-Z_{1}Z_{\oneb}$ acts on each $H_{p,q}$ by multiplication by $-q(p+1)$. (A quick way to see this is to note that $Z_1$ is $\mathrm{SU}(2)$-invariant and hence so is $\square_b$; since $\square_b$ preserves each $H_{p,q}$, Schur's lemma tells us that $\square_b$ must act by a constant on each $H_{p,q}$ and that constant is easily found by testing $\square_b$ on the element $(z+w)^p(\bar{z}-\bar{w})^q$ of $H_{p,q}$.) Thus, by definition, $\mathcal{Q}_0$ acts by $0$ on $H_{p,0}$ and by $-\frac{1}{q(p+1)}$ on $H_{p,q}$ for $q>0$. From this and the equivalence of the $H^s_{FS}$-norm with \eqref{eq:FS-norm} it follows that $\mathcal{Q}_0$ gains two derivatives in Folland-Stein spaces (indeed, $\mathcal{Q}_0$ is a Heisenberg pseudodifferential operator of order $-2$ \cite{BurnsEpstein1990b}). Now, if $\mathcal{S}_0$ denotes the orthogonal projection from $L^2$ functions on $S^3$ to CR functions with respect to the standard CR structure on $S^3$ then, by construction, $-Z_{1}Z_{\oneb}\mathcal{Q}_0 = \mathrm{Id} - \mathcal{S}_0$ and so $-Z_{\oneb}\mathcal{Q}_0$ the (unique) partial inverse to $Z_1$ that is zero on $\mathrm{ker}\,Z_{\oneb}$. By considering the action on each $H_{p,q}$ (or by noting that $\overline{\square}_bZ_{\oneb}=Z_{\oneb}\square_b$) it easy to check that $Z_{\oneb}\mathcal{Q}_0=\overline{\mathcal{Q}_0}Z_{\oneb}$. 
From the above discussion it also readily follows that $(Z_{\oneb}\mathcal{Q}_0)^2=(\overline{\mathcal{Q}_0}Z_{\oneb})^2$ is the (unique) partial inverse to $(Z_1)^2$ that is zero on $\mathrm{ker}\,(Z_{\oneb})^2$. In particular, we have the following lemma.

\begin{lemma}\label{lem:Q_0}
If $g\in H^s_{FS}$ and $\mathcal{P}_2 g = 0$ then $u = (\overline{\mathcal{Q}_0}Z_{\oneb})^2 g$ solves $(\nabla_1)^2 u=g$. Moreover, there is a constant $C$ depending only on $s$ such that
\beq
\norm u \norm_{s+2} \leq C \norm g \norm_{s}.
\eeq
\end{lemma}
\begin{proof}
Recalling that $\mathcal{P}_2 g = 0$ means that $g$ has vanishing components in $H_{p,q}$ for $q=0,1$ (and hence is in the image of $(\nabla_1)^2$) this follows from the definition of $\mathcal{Q}_0$ and the fact that it gains two derivatives in Folland-Stein space as discussed in the paragraph before the lemma.
%
\end{proof}


\begin{proof}[Proof of Proposition \ref{prop:convergence-ft}]
By the construction in the proof of Proposition \ref{prop:formalSol} and an induction using Lemma \ref{lem:BE-induction} we obtain \eqref{eqn:BE-tphidot} (Lemma \ref{lem:BE-induction} tells us that $A_k$ in the proof of Proposition \ref{prop:formalSol} will be in $\mathfrak{D}_{BE}$ and the $B_k$ will be zero, for each $k\geq 1$). It remains to be shown that $f_t$ is analytic in $t$, when viewed as taking values in the Banach space $H^s_{FS}$, and that for each fixed $t$ the function $f_t$ is $C^{\infty}$. Writing $\dot{\varphi}= \dot{\varphi}_{11}$, we then have
$$
L^{(1)} = \vfd\nabla_1\nabla_{\oneb} +\vfd\nabla_{\oneb}\nabla_1+(\nabla_1\vfd)\nabla_{\oneb}-
(\nabla_{\oneb}\vfd)\nabla_1-i\nabla_0\vfd,
$$
$L^{(2)}= \vfd^2 (\nabla_{\oneb})^2$, and $L^{(j)}=0$ for $j\geq 3$. Thus, the $f^{(k)}$ are determined by the equations $(\nabla_1)^2f^{(0)} = \vfd$, $(\nabla_1)^2f^{(1)} = -L^{(1)}f^{(0)}$,
$$
(\nabla_1)^2f^{(k)} = -L^{(1)}f^{(k-1)} - L^{(2)} f^{(k-2)}, \quad k\geq 2,
$$
and the fact that they are orthogonal to the kernel of $(\nabla_1)^2$. By Lemma \ref{lem:Q_0} we have $f^{(0)} = (\overline{\mathcal{Q}_0}Z_{\oneb})^2 \vfd$, $f^{(1)} = - (\overline{\mathcal{Q}_0}Z_{\oneb})^2L^{(1)}f^{(0)}$ and
$$
f^{(k)} = - (\overline{\mathcal{Q}_0}Z_{\oneb})^2\left(L^{(1)}f^{(k-1)} + L^{(2)} f^{(k-2)}\right), \quad k\geq 2.
$$
It follows that $\norm f^{(0)}\norm_{s} \leq C \norm \vfd\norm_{s-2} \leq C \norm \vfd\norm_{s}$ and, using the expressions for $L^{(1)}$ and $L^{(2)}$ above,
\beq
\norm f^{(1)}\norm_{s} \leq 5C \norm\vfd\norm_{s} \norm f^{(0)}\norm_{s}
\eeq
and
\beq
\norm f^{(k)}\norm_{s} \leq C \left(5\norm\vfd\norm_{s} \norm f^{(k-1)}\norm_{s} + \norm\vfd\norm_{s}^2 \norm f^{(k-2)}\norm_{s} \right), \quad k\geq 2.
\eeq
Choose $C_s$ such that $C_s^2 > C(5C_s+1)$. By induction using the above display we then conclude $\norm f^{(k)}\norm_s \leq (C_s \norm \vfd\norm_s)^{k+1}$ for all $k\in\{0,1,2,\ldots\}$. This proves that $f_t =\sum_{k=0}^\infty f^{(k)}t^k$ converges for $|t|< (C_s\norm\dot{\varphi}_{11} \norm_s )^{-1}$ to an analytic function valued in $H_{FS}^s$ functions on $S^3$.

To complete the proof, we shall now show that for each fixed $t$, $|t|< (C_s\norm\dot{\varphi}_{11} \norm_s )^{-1}$, $f_t$ is $C^\infty$ smooth. This follows by an elliptic regularity argument parallel to that given in the proof of Theorem 5.3 in \cite{BurnsEpstein1990b}. Fix $s_0\geq 10$ and $t$ with $|t|< (C_{s_0}\norm\dot{\varphi}_{11} \norm_{s_0} )^{-1}$, and let $f=f_t$ and $\varphi=t\vfd$. By construction $f$ is orthogonal to the kernel of $(\nabla_1)^2$ and satisfies $(\nabla_1)^2f  =- L_{\varphi} f + \vfd$. Applying $(\overline{\mathcal{Q}_0}Z_{\oneb})^2$ to this last equation we get
$$
f = -  (\overline{\mathcal{Q}_0}Z_{\oneb})^2L_{\varphi} f + (\overline{\mathcal{Q}_0}Z_{\oneb})^2\vfd.
$$
Letting $A = -  (\overline{\mathcal{Q}_0}Z_{\oneb})^2L_{\varphi}$ we then have
\beq \label{eqn:(I-A)f}
(I-A)f = (\overline{\mathcal{Q}_0}Z_{\oneb})^2\vfd.
\eeq
Since $\overline{\mathcal{Q}_0}\in Op\,S^{-2}_{\mathcal{V}}$ (in the notation of the Heisenberg pseudodifferential calculus of Beals and Greiner \cite{BealsGreiner1988}), it is easy to see that $I-A\in Op\,S^{0}_{\mathcal{V}}\subset Op\,S^{0}_{\frac{1}{2},\frac{1}{2}}$; here we are using the notation $Op\,S^{m}_{\frac{1}{2},\frac{1}{2}}$ for the classical pseudodifferential operators of type $(\frac{1}{2},\frac{1}{2})$ and order $m$. As in \cite{BurnsEpstein1990b} (where $A$ is taken to be $\mathcal{Q}_0 Z_1 \overline{\varphi} Z_1$) it is easy to see that if $\norm \vfd\norm_{L^{\infty}(S^3)}$ is sufficiently small, then the principal symbol of $I-A$ is positive. If we further take $\vfd$ to be sufficiently small in $C^1$ then the argument on pages 832-833 of \cite{BurnsEpstein1990b} shows that there is a constant $K_s$ (depending on the $H^s$-norm of $\vfd$) such that
\beq
\norm u\norm_{H^s} \leq K_s \norm (I-A)u\norm_{H^s}.
\eeq
Applying this to \eqref{eqn:(I-A)f} we have
\beq
\norm f\norm_{H^s} \leq K_s \norm(\overline{\mathcal{Q}_0}Z_{\oneb})^2\vfd\norm_{H^s}  \leq K_s' \norm\vfd\norm_{H^{s-1}}
\eeq
where in the last inequality we have used that $(\overline{\mathcal{Q}_0}Z_{\oneb})^2 \in Op\, S^{-2}_{\mathcal{V}} \subset Op\, S^{-1}_{\frac{1}{2},\frac{1}{2}}$. Since $\vfd\in C^{\infty}(S^3,\mathbb{C})$ it follows that $f\in C^{\infty}(S^3,\mathbb{C})$.
\end{proof}
\begin{remark}
In the above regularity argument we used standard Sobolev spaces as in pages 832-833 of \cite{BurnsEpstein1990b}. In the proof of Theorem \ref{thm:smooth-soln} below it will be convenient to instead work with Folland-Stein spaces and we will see that this is possible using results of Ponge \cite{Ponge2008} on the Heisenberg pseudodifferential calculus (so we could have used Folland-Stein spaces in the above proof). However, in Section \ref{sec:proofs-slice-thms} where we prove the slice theorems we shall again work with standard Sobolev spaces (for consistency with \cite{ChengLee1995}) and will again need an elliptic regularity argument of the kind given above. 
\end{remark}

In order to apply Theorem \ref{P:embed} to produce a family of embeddings realizing a family of deformations $t\vfd$ with $\vfd \in \mathfrak{D}_{BE}$ we need to ensure that the family of solutions $f_t$ to the tangency equation are such that $\re f_t$ has strict sign. To do this we modify Proposition \ref{prop:convergence-ft} making use of the freedom to add a constant to $f_0 = f^{(0)}$ due to the kernel of $(\nabla_1)^2$ and prove the following theorem.
\begin{theorem}\label{thm:lambda-convergence}
For any $s \geq 10$, $\lambda \in \mathbb{R}$, $R>0$ there exists $\epsilon>0$ such that if $\dot{\varphi}_{11}\in \mathfrak{D}_{BE}$ satisfies $\norm \vfd \norm_s < \epsilon$ then there is a unique formal power series $f_t =\lambda + \tilde{f}_t = \lambda + \sum_{k=0}^\infty \tilde{f}^{(k)}t^k$ such that
\begin{itemize}
\item[(i)] $\tilde{f}^{(k)} \in  (Z_{\oneb})^2 \mathfrak{D}_{BE}$ for $k\geq 0;$
\item[(ii)] $f_t$ converges for $|t|<R$ to an analytic function taking values in $H_{FS}^s$, and for each fixed $t$, $|t|< R$, $f_t$ is a $C^{\infty}$ function on $S^3;$
\item[(iii)] $f_t$ solves
\beq \label{eqn:linearized-emb-t}
\nabla_1^t\nabla_1^t f_t + iA_{11}(t)f_t = \frac{\dot\vf_{11}}{1-|\vf(t)|^2}
\eeq
where $\varphi_{11}(t) = t\vfd_{11};$
\item[(iv)] if $\lambda\neq 0$ and $R$ is sufficiently large, then $\mathrm{Re}\,f_t$ has a strict sign for $|t|\leq 1$.
\end{itemize}
\end{theorem}
\begin{proof}
Recall that the equation \eqref{eqn:linearized-emb-t} for $f_t$ is equivalent to the equation
$$
(\nabla_1)^2 f_t + L_{\varphi}f_t = \vfd_{11}.
$$
In terms of $\tilde{f}_t$ this equation takes the form
$$
(\nabla_1)^2 \tilde{f}_t + L_{\varphi}\tilde{f}_t = \vfd_{11} + it\lambda  \nabla_0\vfd_{11}.
$$
Formally this equation is equivalent to $(\nabla_1)^2\tilde{f}^{(0)} = \vfd_{11}$, $(\nabla_1)^2\tilde{f}^{(1)} = -L^{(1)}\tilde{f}^{(0)} +i\lambda  \nabla_0\vfd_{11}$,
$$
(\nabla_1)^2\tilde{f}^{(k)} = -L^{(1)}\tilde{f}^{(k-1)} - L^{(2)} \tilde{f}^{(k-2)}, \quad k\geq 2.
$$
Since $\dot{\varphi}_{11}\in \mathfrak{D}_{BE}$ implies $\nabla_0\dot{\varphi}_{11}\in \mathfrak{D}_{BE}$, the unique formal solvability of this equation for $\tilde{f}_t= \sum_{k=0}^\infty \tilde{f}^{(k)}t^k$ satisfying (i) follows easily by induction using Lemma \ref{lem:BE-induction} as in the proof of Proposition \ref{prop:convergence-ft}. Arguing as in the proof of Proposition \ref{prop:convergence-ft} we obtain
$\norm \tilde{f}^{(0)}\norm_{s} \leq C \norm \vfd\norm_{s-2} \leq C \norm \vfd\norm_{s}$,
\beq
\norm \tilde{f}^{(1)}\norm_{s} \leq C\left(5 \norm\vfd\norm_{s} \norm \tilde{f}^{(0)}\norm_{s} + |\lambda|  \norm\vfd\norm_{s}\right)
\eeq
and
\beq
\norm \tilde{f}^{(k)}\norm_{s} \leq C \left(5\norm\vfd\norm_{s} \norm \tilde{f}^{(k-1)}\norm_{s} + \norm\vfd\norm_{s}^2 \norm \tilde{f}^{(k-2)}\norm_{s} \right), \quad k\geq 2.
\eeq
If we choose a constant $C_s$ such that $C_s\geq C(5C+|\lambda|)$ and $C_s^2 \geq C(5C_s + 1)$, then it follows by induction that $\norm f^{(k)}\norm_s \leq (C_s \norm \vfd\norm_s)^{k}$ for all $k\in\{1,2,\ldots\}$. Moreover, as long as $C_s$ is also $\geq C$, we have $\norm \tilde{f}^{(0)}\norm_{s}\leq C_s \norm \vfd\norm_{s}$. If such a $C_s$ has been chosen, the radius of convergence of the power series is at least $\rho=(C_s\norm \vfd\norm_{s})^{-1}$. Thus we have proved that for any $s \geq 10$, $\lambda \in \mathbb{R}$, $R>0$ there exists $\epsilon>0$ (e.g., $\epsilon=(C_sR)^{-1}$) such that if $\dot{\varphi}_{11}\in \mathfrak{D}_{BE}$ satisfies $\norm \vfd \norm_s < \epsilon$ then there is a unique formal power series $f_t =\lambda + \tilde{f}_t = \lambda + \sum_{k=0}^\infty \tilde{f}^{(k)}t^k$ satisfying $\tilde{f}^{(k)} \in  (Z_{\oneb})^2 \mathfrak{D}_{BE}$ for $k\geq 0$ such that $f_t$ converges for $|t|<R$ to an analytic function taking values in $H_{FS}^s$; by construction $f_t$ solves \eqref{eqn:linearized-emb-t}. The $C^{\infty}$ smoothness of $\tilde{f}_t$, and hence of $f_t$, for fixed $t$ follows as in the proof of Proposition \ref{prop:convergence-ft}.

We note that from the construction of $f_t$ above and the Sobolev embedding theorem for Folland-Stein spaces (since $s$ is $\geq 3$), we have for some constant $C'$:
\begin{align}
\norm f_t -\lambda\norm_{\infty} &\leq C'\norm \tilde{f}_t\norm_s \leq C'\sum_{k=0}^{\infty} \norm\tilde{f}^{(k)}\norm_s |t|^k \leq C'\left(C_s\norm \vfd\norm_s+ \sum_{k=1}^{\infty}  (C_s \norm \vfd\norm_s)^{k} |t|^k\right)   \\ & \leq  C'\left(C_s \norm \vfd\norm_s+ \frac{C_s \norm \vfd\norm_s|t|}{1-|t|C_s \norm \vfd\norm_s}\right) 
\leq C'\left(\frac{1}{R}+ \frac{|t|}{R-|t|}\right).\nonumber
\end{align}
From this it is easy to see that if $\lambda\neq 0$ and $R$ is sufficiently large, then $\mathrm{Re}\,f_t$ has a strict sign for $|t|\leq 1$.
\end{proof}

Theorem \ref{thm:BE-parametric} now follows from Theorems \ref{thm:lambda-convergence} and \ref{P:embed}.

\subsection{Families of embeddable deformations with general linearized term}

We now return to the case of general embeddable infinitesimal deformations $\vfd_{11}$, for which analyticity in $t$ of our formal solution may no longer hold. Our aim is to prove the following theorem.

\begin{theorem}\label{thm:smooth-soln}
For any $s \geq 10$, $\lambda < 0$, $T>0$ there exists $\epsilon>0$ such that if $\dot{\varphi}_{11}\in \mathfrak{D}_0$ satisfies $\norm \vfd \norm_s < \epsilon$ then there are unique $f_t = \lambda + \tilde{f}_t \in C^{\infty}([0,T]\times S^3,\mathbb{C})$ and $\varphi_{11}(t) \in C^{\infty}([0,T]\times S^3,\mathbb{C})$ such that
\begin{itemize}
\item[(i)] $\tilde{f}_t \in  (Z_{\oneb})^2 \left( C^{\infty}(S^3,\mathbb{C})\right)$ for all $t\in [0,T];$
\item[(ii)] $\varphi_{11}(t) = t\vfd_{11}+\mu_{11}(t)$ where $\mu_{11}(t)\in \mathfrak{D}_0^{\perp}$ for all $t\in [0,T]$ and $\mu_{11}(t) = O(t^2);$
\item[(iii)] $f_t$ and $\varphi_{11}(t)$ solve the tangency equation \eqref{Nabla2ft}.
\end{itemize}
\end{theorem}

In the proof of Theorem \ref{thm:smooth-soln} we will solve \eqref{Nabla2ft} by splitting it into two equations using the $L^2$ orthogonal projection $\mathcal{P}_1$ from $\mathfrak{D}$ onto $\mathfrak{D}_0$ and the complementary projection $\mathcal{P}_2$ onto $\mathfrak{D}_0^{\perp}$. When we apply $\mathcal{P}_1$ to \eqref{Nabla2ft} we get an equation with main term $(\nabla_1)^2 f_t$, and when we apply $\mathcal{P}_2$ we get an equation with main term $\left(\frac{d}{dt}+i\lambda \nabla_0 \right) \mathcal{P}_2 \vf$ where $i\nabla_0$ acts on the deformation tensor $\mathcal{P}_2 \vf$ as $iT-4$ (since $\omega_1{}^1{}_0 = -iR = -2i$). A crucial point in the proof is that by \eqref{eqn:Reeb-spherical-harmonics} the operator $-iT$ when applied to elements of $\mathfrak{D}_0^{\perp}$ behaves like a second order subelliptic operator. For the proof of Theorem \ref{thm:smooth-soln} we will need the following lemmas which are based on this observation. The first lemma is an immediate consequence of the action of the Reeb vector field (cf.\ \eqref{eqn:Reeb-spherical-harmonics}) and the sublaplacian on the spherical harmonics and the description of the Folland-Stein spaces in terms of spherical harmonic decompositions.

\begin{lemma}\label{lem:energy-estimate1}
For any $\gamma>1$ and $s>0$ there exists $\beta_s>0$ such that the operator $-iT+\gamma$ satisfies the following estimate
\beq\label{eq:Reeb-energy-estimate}
\beta_s\norm u \norm_{s+2}
\leq \norm(-iT+\gamma) u \norm_{s}
\eeq
for any $u\in H^{s+2}_{FS}$ such that $u=\mathcal{P}_2 u$. If $\gamma \geq 2$, then $\beta_s$ can be taken to be $1/3$.
\end{lemma}

In the following we denote by $H^k([0,T_0];\mathcal{B})$ the Sobolev space of functions on $[0,T_0]$ taking values in a given Banach space $\mathcal{B}$, cf.\ \cite[Chapter 5]{Evans2010}.
\begin{lemma}\label{lem:Reeb-heat-equation}
Let $\lambda<0$, $s\geq 0$. If $g \in \bigcap_{k=0}^s H^k([0,T_0];H_{FS}^{2s-2k})$ satisfies $g=\mathcal{P}_2 g$ then there is a unique solution $u \in \bigcap_{k=0}^{s+1} H^k([0,T_0];H_{FS}^{2s+2-2k})$ to
\beq
\left\{
\begin{aligned}
\left(\tfrac{d}{dt} +  \lambda (iT-4) \right) u &= g\\
u(0) &= 0.
\end{aligned}
\right.
\eeq
Moreover, there is a constant $C$ depending only on $\lambda$ and $s$ such that
\beq\label{eqn-parabolic-estimate}
\sum_{k=0}^{s+1} \bignorm \frac{d^k u}{dt^k} \bignorm_{L^2([0,T_0]; H_{FS}^{2s+2-2k})} \leq C \sum_{k=0}^{s} \bignorm \frac{d^k g}{dt^k}\bignorm_{L^2([0,T_0]; H_{FS}^{2s-2k})}.
\eeq
\end{lemma}
\begin{proof}
The existence of a unique solution follows immediately by decomposing writing $g$ in terms of spherical harmonics and using that $iT$ acts by a constant on each $H_{p,q}$ (see \eqref{eqn:Reeb-spherical-harmonics}); indeed, if $g=\sum g_{p,q}$ (with $p=0,1,2,\ldots$ and $q=0,1$ since $g=\mathcal{P}_2 g$) then the explicit solution is given by
\beq
u(t) = \sum_{p,q}  \int_0^te^{-c_{p,q}(t-\tau)}g_{p,q}(\tau)d\tau 
\eeq
where $c_{p,q}=\lambda(q-p-4)$ (note that $c_{p,q}>0$, since $\lambda<0$ and $q$ equals $0$ or $1$). It remains to establish the estimate \eqref{eqn-parabolic-estimate}. To establish \eqref{eqn-parabolic-estimate} for  $s=0$ we take an $L^2(S^3)$-inner product of the equation with $u'=\frac{d}{dt}u$ to obtain
\beq
\norm u'\norm^2_{H^0_{FS}} + \sum_{p,q} c_{p,q}u_{p,q}\overline{u'_{p,q}} = \sum_{p,q} g_{p,q}\overline{u'_{p,q}}.
\eeq
Taking the real part we obtain
\beq
\norm u'\norm^2_{H^0_{FS}} + \frac{1}{2}\frac{d}{dt}\sum_{p,q} c_{p,q}|u_{p,q}|^2 = \re\sum_{p,q} g_{p,q}\overline{u'_{p,q}}.
\eeq
Hence, using the big constant-small constant inequality to estimate the left hand side from above by $\frac{C}{2}\norm g\norm^2_{H^0_{FS}} + \frac{1}{2}\norm u'\norm^2_{H^0_{FS}}$, we obtain
\beq
\norm u'\norm^2_{H^0_{FS}} + \frac{d}{dt}\sum_{p,q} c_{p,q}|u_{p,q}|^2 \leq C\norm g\norm^2_{H^0_{FS}}.
\eeq
Integrating over $[0,T_0]$ we obtain
\beq\label{eqn-parabolic-estimate-base-case}
\norm u'\norm^2_{L^2([0,T_0]; H_{FS}^{0})} + \sum_{p,q} c_{p,q}|u_{p,q}(T_0)|^2 \leq C\norm g\norm^2_{L^2([0,T_0]; H_{FS}^{0})}.
\eeq
Since $\lambda<0$ (and hence $c_{p,q}>0$) we have $\norm u'\norm^2_{L^2([0,T_0]; H_{FS}^{0})} \leq C\norm g\norm^2_{L^2([0,T_0]; H_{FS}^{0})}$. Thus we have gained in temporal regularity. To show that we also gain two Folland-Stein derivatives we rewrite the equation as $\lambda(iT-4)u=g-u'$ and use the result that we have just established together with the fact that $\lambda(iT-4)$ has a well-defined inverse that gains two derivatives in Folland-Stein spaces when acting on $\overline{\bigoplus_{q\in\{0,1\}}H_{p,q}}$ (where the overline denotes $L^2$-closure), see Lemma \ref{lem:energy-estimate1} (taking $\gamma=4$). Hence, for a possibly larger constant $C$,
\beq
\norm u'\norm^2_{L^2([0,T_0]; H_{FS}^{0})} + \norm u\norm^2_{L^2([0,T_0]; H_{FS}^{2})}  \leq C\norm g\norm^2_{L^2([0,T_0]; H_{FS}^{0})}.
\eeq
This establishes the estimate \eqref{eqn-parabolic-estimate} for $s=0$ and forms the base case for an inductive proof of \eqref{eqn-parabolic-estimate} for integers $s\geq 0$. Let $s\geq 0$ and suppose \eqref{eqn-parabolic-estimate} is known for this $s$. Let $g \in \bigcap_{k=0}^{s+1} H^k([0,T_0];H_{FS}^{2s-2k})$. Since $\left(\tfrac{d}{dt} +  \lambda (iT-4) \right) u' = g'$ we may apply the estimate known for this $s$ to $u'$ and $g'$ to obtain
\beq\label{eqn-parabolic-estimate-primed}
\sum_{k=1}^{s+2} \bignorm \frac{d^k u}{dt^k} \bignorm_{L^2([0,T_0]; H_{FS}^{2s+4-2k})} \leq C \sum_{k=1}^{s+1} \bignorm \frac{d^k g}{dt^k}\bignorm_{L^2([0,T_0]; H_{FS}^{2s+2-2k})}.
\eeq
We can clearly replace $k=1$ by $k=0$ on the right hand side, and the missing term $\norm u\norm_{L^2([0,T_0]; H_{FS}^{2s+4})}$ on the right hand side can then be estimated using that $\lambda(iT-4)u=g-u'$ and the estimate for $u'$ as previously. In this way we obtain, for a possibly larger $C$, that 
\beq\label{eqn-parabolic-estimate-s-plus-1}
\sum_{k=0}^{s+2} \bignorm \frac{d^k u}{dt^k} \bignorm_{L^2([0,T_0]; H_{FS}^{2s+4-2k})} \leq C \sum_{k=0}^{s+1} \bignorm \frac{d^k g}{dt^k}\bignorm_{L^2([0,T_0]; H_{FS}^{2s+2-2k})},
\eeq
as required. The result follows by induction.
\end{proof}
\begin{remark}
Note that there is of course nothing particularly special about the number $-4$ in the lemma, which could be replaced by any nonpositive real number. As explained above, this is the result that we will need in the proof of Theorem \ref{thm:smooth-soln}. 
\end{remark}

In practice, we will need to allow for small perturbations of the operator $\lambda (iT-4)$ in Lemma \ref{lem:Reeb-heat-equation}. Recall \cite{Ponge2008} that a Heisenberg pseudodifferential operator $P$ of order $m$ ($m$ real) on a compact manifold maps $H^s_{FS}$ continuously to $H^{s-m}_{FS}$. Note that, with respect to a (locally defined) $S^1$-invariant frame $W_1$, a second order Heisenberg differential operator $P$ can be written as
\beq
P = a^{11}W_1W_1 + b^{\oneb\oneb}W_{\oneb}W_{\oneb} + c^{1\oneb}W_1 W_{\oneb} + d^{\oneb 1} W_{\oneb}W_1 + e^1 W_1 + f^{\oneb} W_{\oneb} + g.
\eeq
In such a frame the operation of taking a commutator with $T$ is equivalent to differentiating the coefficients with respect to $T$, that is,
\begin{multline}
[T,P] =  (Ta^{11})W_1W_1 + (Tb^{\oneb\oneb})W_{\oneb}W_{\oneb} + (Tc^{1\oneb})W_1 W_{\oneb} \\+ (Td^{\oneb 1}) W_{\oneb}W_1 + (Te^1) W_1 + (Tf^{\oneb}) W_{\oneb} + Tg.
\end{multline}
In particular, $[T,P]$ is also a second order Heisenberg differential operator. Now, let $P$ be a Heisenberg pseudodifferential operator of order 2 on $S^3$. By decomposing $P$ with respect a finite open cover of $S^3$ (by $S^1$-invariant sets) such that on each open set we have a fixed $S^1$-invariant frame for $T^{1,0}$  and using a partition of unity we can show that $[T,P]$ is also a second order pseudodifferential operator, and hence is bounded from $H^s_{FS}$ to $H^{s-2}$ for each $s$. 

In the following lemma we modify Lemma \ref{lem:Reeb-heat-equation} by adding to $\lambda (iT-4)$ a perturbation term of the form $\mathcal{P}_2 \circ P$ with $P$ a small time-dependent second order Heisenberg pseudodifferential operator (recall that $\lambda (iT-4)$ is a second order Heisenberg differential operator itself and is subelliptic when restricted to the band of functions $u$ with $u=\mathcal{P}_2u$); the higher order estimates then depend on the corresponding norms of $\mathcal{P}_2\circ P$. We do this by adapting Hamilton's proof of Lemmas 6.9 and 6.10 in \cite{Hamilton1982}. Hence, if $P=P_t$, $t\in [0,T_0]$, is a smooth one-parameter family of second order Heisenberg pseudodifferential operators we define $[P]_{2s}$ by
\beq \label{eq:spatial-norm-on-P}
 [P]_{2s} = \int_0^{T_0} \norm\, [iT-4, \mathcal{P}_2\circ P_t]\, \norm_{H_{FS}^{2s+2}\to H_{FS}^{2s}}^2 \,dt
\eeq
and $|[P]|_{2s}$ by
\beq
|[P]|_{2s} = \sum_{k=0}^s \left[ \frac{d^k}{d t^k} P\right]_{2s-2k}.
\eeq
The norms are constructed precisely so that we will be able to apply the operator $T$ (or better, $iT-4$) to our equation in order to gain spatial regularity (and then gain temporal regularity by shifting all the spatial derivatives to the right hand side). In detail:

\begin{lemma}\label{lem:perturbed-Reeb-heat-equation}
Let $\lambda<0$, $s\geq 0$. Let $P=P_t$, $t\in [0,T_0]$, be a smooth one-parameter family of second order Heisenberg pseudodifferential operators on $S^3$ and suppose
\beq \label{eq:small-perturbation}
\norm P u \norm_{0} \leq \frac{|\lambda|}{10}\norm u \norm_{2} 
\eeq
for all $u\in H^2_{FS}$ and all $t$. If $g \in \bigcap_{k=0}^s H^k([0,T_0];H_{FS}^{2s-2k})$ satisfies $g=\mathcal{P}_2 g$ then there is a unique solution $u = \mathcal{P}_2 u \in \bigcap_{k=0}^{s+1} H^k([0,T_0];H_{FS}^{2s+2-2k})$ to
\beq
\left\{
\begin{aligned}
\left(\tfrac{d}{dt} + \lambda (iT-4) + \mathcal{P}_2\circ P\right) u &= g\\
u(0) &= 0.
\end{aligned}
\right.
\eeq
Moreover, there is a constant $C$ such that
\beq\label{inequality-u-g-with-Pt}
\sum_{k=0}^{s+1} \bignorm \frac{d^k u}{dt^k} \bignorm_{L^2([0,T_0]; H_{FS}^{2s+2-2k})} \leq C \sum_{k=0}^{s} \bignorm \frac{d^k g}{dt^k}\bignorm_{L^2([0,T_0]; H_{FS}^{2s-2k})} + C|[P]|_{2s}\norm g\norm_{L^2([0,T_0];H_{FS}^0)},
\eeq
where $C$ depends only on $\lambda$ and $s$.
\end{lemma}
\begin{proof}
The existence of a unique weak ($L^2$) solution follows by using Galerkin approximations for the equation decomposed into spherical harmonics and establishing convergence using energy estimates similar to the base case of the higher regularity argument below (cf.\ \cite[Section 7.1]{Evans2010}). To establish the base case of the \emph{a priori} estimate we argue as in the proof of Lemma \ref{lem:Reeb-heat-equation} to obtain the following modified version of \eqref{eqn-parabolic-estimate}:
\begin{multline}
\norm u'\norm^2_{L^2([0,T_0]; H_{FS}^{0})} + \sum_{p,q} c_{p,q}|u_{p,q}(T_0)|^2  \leq  C\norm g\norm^2_{L^2([0,T_0]; H_{FS}^{0})} - 2\re\int_0^{T_0}\int_{S^3}\overline{u'(t)} P u(t)dt.
\end{multline}
Here $c_{p,q}=\lambda(q-p-4)>0$ is the multiplier corresponding to $\lambda(iT-4)$. Throwing away the second term on the left and then estimating the last term on the right above by $\frac{1}{2}\norm u'\norm^2_{L^2([0,T_0]; H_{FS}^{0})} + 2\norm P u\norm^2_{L^2([0,T_0]; H_{FS}^{0})}$ and using the assumption on the norm of $P$ we obtain
\beq\label{estimate-u-H1}
\frac{1}{2}\norm u'\norm^2_{L^2([0,T_0]; H_{FS}^{0})}  \leq  C\norm g\norm^2_{L^2([0,T_0]; H_{FS}^{0})} + \frac{|\lambda|^2}{25}\norm u \norm_{L^2([0,T_0]; H_{FS}^{2})}^2. 
\eeq
By assumption, the perturbation term $P$ is small enough such that $\lambda (iT-4) + \mathcal{P}_2\circ P$ is still invertible and its inverse gains two Folland-Stein derivatives when acting on $L^2$-functions whose spherical harmonic decomposition is supported in the $H_{p,q}$ for $q=0,1$.
Writing the equation in the form $\left(\lambda (iT-4) + P\right)u=g-u'$ and using Lemma \ref{lem:energy-estimate1}, \eqref{eq:small-perturbation} and \eqref{estimate-u-H1}, a straightforward argument shows that
\beq
\frac{27|\lambda|}{30}\norm u\norm_{L^2([0,T_0]; H_{FS}^{2})} \leq C \norm g\norm_{L^2([0,T_0]; H_{FS}^{0})} + \frac{\sqrt{2}|\lambda|}{5}\norm u\norm_{L^2([0,T_0]; H_{FS}^{2})},
\eeq
for some constant $C$. It follows that $\norm u\norm_{L^2([0,T_0]; H_{FS}^{2})} \leq C \norm g\norm_{L^2([0,T_0]; H_{FS}^{0})}$ for some new constant $C$ that depends on $\lambda$. Combined with \eqref{estimate-u-H1} this establishes the base ($s=0$) case for an inductive proof of \eqref{inequality-u-g-with-Pt}. 

The inductive step is handled in a similar manner to the proof of Lemmas 6.9 and 6.10 in Hamilton's Ricci flow paper \cite{Hamilton1982} except with ordinary $L^2$-based Sobolev spaces replaced by $L^2$-based Folland-Stein spaces; these proofs are similar to the argument in Lemma \ref{lem:Reeb-heat-equation} except that Hamilton differentiates the equation with respect to space to gain spatial derivatives first and then uses the equation to gain time derivatives (rather than the other way around). Our definition \eqref{eq:spatial-norm-on-P} was made so that we avoid the need for interpolation in the spatial regularity part of the argument (since we can merely break $(iT-4)\mathcal{P}_2 P u$ up into $\mathcal{P}_2 P (iT-4)u + [(iT-4), \mathcal{P}_2\circ P]u$ and then estimate these two terms) and therefore avoid working directly with the symbol of the operator $P$ (with respect to a local $S^1$-invariant frame); we note however that the Folland-Stein version of the interpolation inequality used by Hamilton in the spatial regularity part does hold, e.g., by \cite[Corollary 2.12]{ChengLee1995}. For the temporal regularity part one invokes the interpolation inequality for ordinary Sobolev spaces (with respect to the time variable) in order to reduce to the cases when all the time derivatives fall on $P$ or all on $u$ when differentiating the equation satisfied by $u$; the argument is then as in \cite[Lemma 6.10]{Hamilton1982}. The details are left to the reader.
\end{proof}

To prove Theorem \ref{thm:smooth-soln} we shall work with graded Fr\'echet spaces where the grading comes from Folland-Stein norms on functions on $S^3$. To this end we observe the following.

\begin{theorem}\label{thm:tameness-of-FS}
Endow $C^{\infty}(S^3,\mathbb{C})$ with the structure of a graded Fr\'echet space using the Folland-Stein norms $\norm\,\cdot\,\norm_s$ for all $s\in \{0,1,2,\ldots\}$. With this structure $C^{\infty}(S^3,\mathbb{C})$ is a tame Fr\'echet space.
\end{theorem}
\begin{remark}
Note that the tame structure on $C^{\infty}(S^3,\mathbb{C})$ defined in Theorem \ref{thm:tameness-of-FS} is not equivalent to the standard one induced by the scale of standard ($L^2$) Sobolev norms.
\end{remark}
\begin{proof}
We need to establish the existence of tame linear maps $L:C^{\infty}(S^3,\mathbb{C}) \to \Sigma (\mathcal{B})$ and $M:\Sigma (\mathcal{B})\to C^{\infty}(S^3,\mathbb{C})$ such that $ML=\mathrm{id}$, where $\mathcal{B}$ is some Banach space and $\Sigma (\mathcal{B})$ is the Fr\'echet space of exponentially decreasing sequences in $\mathcal{B}$, i.e.\ the space of sequences $(a_n)$ such that $\norm (a_n)\norm_{\Sigma (\mathcal{B}),s}^2 = \sum_{n=0}^{\infty} 2^{sn} \norm a_n \norm_{\mathcal{B}}^2<\infty$ for all $s$ endowed with the scale of norms $\norm \,\cdot\,\norm_{\Sigma (\mathcal{B}),s}$ \cite{Hamilton1982BAMS}. We take $\mathcal{B}=L^2(S^3,\mathbb{C})$. Expanding an element $u \in C^{\infty}(S^3,\mathbb{C})$ in spherical harmonics as $u=\sum_{p,q\geq 0} u_{p,q}$ we let $u_n$ denote the sum of the $u_{p,q}$ where $2^n \leq 2pq+p+q+1 < 2^{n+1}$ when $n\geq 1$ and let $u_0=u_{0,0}$. Then we define the map $L:C^{\infty}(S^3,\mathbb{C}) \to \Sigma (\mathcal{B})$ by $L(u) = (u_n)$. Then
\beq
\begin{aligned}
\norm L(u)\norm_{\Sigma (\mathcal{B}),s}^2 &= \sum_{n=0}^{\infty} 2^{sn} \norm u_n \norm_{L^2}^2 \\
&= \norm u_0\norm_{L^2}^2 + \sum_{n=1}^{\infty}  \sum_{2^n \leq 2pq+p+q+1 < 2^{n+1}} 2^{sn} \norm u_{p,q}\norm_{L^2}^2 \\
&\leq  \norm u_0\norm_{L^2}^2 + \sum_{n=1}^{\infty}  \sum_{2^n \leq 2pq+p+q+1 < 2^{n+1}} (2pq+p+q+1)^s \norm u_{p,q}\norm_{L^2}^2 \\
&=\sum_{p,q\geq 0} (2pq+p+q+1)^s \norm u_{p,q}\norm_{L^2}^2\\
&\leq C_s \norm u\norm_s^2
\end{aligned}
\eeq
for each $s\in \{0,1,2,\ldots\}$. Hence $L$ is tame. To define the map $M$ we first let 
\beq
 \pi_n : L^2(S^3,\mathbb{C}) \to \bigoplus_{2^n \leq 2pq+p+q+1 < 2^{n+1}} H_{p,q}
\eeq
denote the orthogonal projection for $n\geq 1$ and let $\pi_0$ denote the orthogonal projection onto $H_{0,0}$. The map $M$ is then defined by $M((a_n))= \sum_n \pi_n a_n$. It's clear that, by definition, $ML=\mathrm{id}$. Given a sequence $(a_n)\in\Sigma(\mathcal{B})$ we decompose each $a_n$ in spherical harmonics as $a_n = \sum_{p,q} [a_n]_{p,q}$. To see that $M$ is tame we note that $\pi_n a_n$ is orthogonal to $\pi_m a_m$ for $m\neq n$ and hence
\beq
\begin{aligned}
\norm M((a_n))\norm_{s}^2 &= \sum_{n=0}^{\infty} \norm \pi_n a_n \norm_s^2\\
& \leq C_s\left(\norm \pi_0 a_0 \norm_{L^2}^2 + \sum_{n=1}^{\infty} \sum_{2^n \leq 2pq+p+q+1 < 2^{n+1}} (2pq+p+q+1)^s \norm [a_n]_{p,q} \norm_{L^2}^2\right)  \\
& \leq C_s\left(\norm a_0 \norm_{L^2}^2 + \sum_{n=1}^{\infty}  \sum_{2^n \leq 2pq+p+q+1 < 2^{n+1}}  2^{(n+1)s}\norm [a_n]_{p,q} \norm_{L^2}^2\right) \\
& \leq C_s 2^s\left(\norm a_0 \norm_{L^2}^2 + \sum_{n=1}^{\infty} 2^{ns} \norm a_n \norm_{L^2}^2\right) \\
&= C'_s \norm (a_n)\norm_{\Sigma(\mathcal{B}),s}^2,
\end{aligned}
\eeq
for each $s\in\{0,1,2,\ldots\}$. This proves the result.
\end{proof}

Arguing similarly to the above we can endow the space $C^{\infty}(S^3\times[0,T_0],\mathbb{C})$ with a tame Fr\'echet space structure using the norms on the spaces $\bigcap_{k=0}^s H^k([0,T_0];H_{FS}^{2s-2k})$ for each $s\in \{0,1,2,\ldots\}$. Such a tame Fr\'echet space structure was used, e.g., by Hamilton in his Ricci flow paper \cite{Hamilton1982} (to establish short time existence of the flow) except that he used ordinary Sobolev spaces rather than Folland-Stein spaces for the spatial regularity. The idea is to describe the norms equivalently in terms of a dyadic decomposition of the functions expressed in terms of spherical harmonics with respect to the $S^3$ and the Fourier dual variable in the $[0,T_0]$-direction. In more detail one first uses an extension operator to include $C^{\infty}(S^3\times[0,T_0],\mathbb{C})$ into $C^{\infty}_0(S^3\times\mathbb{R},\mathbb{C})$, then one takes the Fourier transform with respect to the $\mathbb{R}$ factor and decomposes with respect to spherical harmonics on $S^3$ to obtain functions of $\omega\in \mathbb{R}$ and $p,q\in \{0,1,2,\ldots\}$ (the extension operator respecting the tame structure can be constructed as in the proof of Corollary 1.3.7 in \cite{Hamilton1982}). The $s$-norms we are using become equivalent to weighted $L^2$ norms with the multiplier $(2pq+p+q+|\omega|+2)^s$, which treats two spatial Folland-Stein derivatives as being on the same level as one time derivative (to see the equivalence with the norm on $\bigcap_{k=0}^s H^k([0,T_0];H_{FS}^{2s-2k})$ for each $s$ consider the binomial expansion of $((2pq+p+q+1) + (|\omega|+1))^{2s}$). The proof of tameness is then highly analogous to the proof of Theorem \ref{thm:tameness-of-FS}, where now we dyadically decompose our functions with respect to $2pq+p+q+|\omega|+2$ rather than $2pq+p+q+1$.

\begin{proof}[Proof of Theorem \ref{thm:smooth-soln}]
We start by recalling that \eqref{Nabla2ft} is equivalent to
\beq
(\nabla_1)^2 f + L_{\varphi}f  = \frac{d}{dt}\varphi
\eeq
where $f=f_t$, $\varphi=\varphi_{11}(t)$ and $L_{\vf}$ is as in \eqref{eq:L}. As before let $\mathcal{P}_1$ denote the $L^2$ orthogonal projection onto the image of $(\nabla_1)^2$ and let $\mathcal{P}_2$ denote the complementary orthogonal projection onto the kernel of $(\nabla_{\oneb})^2$; recall that if $u=\sum_{p,q} u_{p,q}$ is the spherical harmonic decomposition of $u$, then $\mathcal{P}_1(u) = \sum_{p\geq 0,q\geq 2} u_{p,q}$ and $\mathcal{P}_2(u) = \sum_{p\geq 0,q=0,1} u_{p,q}$. It is easy to see that the above displayed equation will hold with $\varphi(t) = t\vfd + \mu(t)$ (and $\mathcal{P}_1\mu(t)=0)$, if and only if
\beq
\begin{aligned}
(\nabla_1)^2 f + \mathcal{P}_1 ( L_{\varphi} f ) &= \vfd \\
\mu' - \mathcal{P}_2 ( L_{\varphi} f ) & = 0
\end{aligned}
\eeq
where $\mu'=\frac{d}{dt}\mu$. Writing $f=\lambda+\tilde{f}$ and noting that $L_{\varphi}f = L_{\varphi}\tilde{f} - i\lambda \nabla_0\varphi$ we may write this system as
\beq\label{eqn-sub-elliptic-parabolic-system}
\begin{aligned}
(\nabla_1)^2 \tilde{f} + \mathcal{P}_1 ( L_{\varphi} \tilde{f})  - \vfd - i\lambda t\nabla_0 \vfd &= 0\\
(\tfrac{d}{dt} + i\lambda\nabla_0)\mu - \mathcal{P}_2 ( L_{\varphi} \tilde{f} )  & = 0.
\end{aligned}
\eeq
Viewing $\vfd$ as a parameter, we now show that \eqref{eqn-sub-elliptic-parabolic-system} can be solved for $\tilde{f}$ and $\mu$ (depending on $\vfd$) near $\vfd=0$. Fix $s_0\geq 10$ and consider the operator $\mathcal{F}$ taking triples $(\tilde{f},\mu,\vfd)$ with
\begin{itemize}
\item $\tilde{f}\in \bigcap_{k=0}^{s_0} H^k([0,T_0];H_{FS}^{2s_0-2k+2}))$ orthogonal to the kernel of $(\nabla_1)^2$, 
\item $\mu=\mathcal{P}_2\mu \in \bigcap_{k=0}^{s_0+1} H^k([0,T_0];H_{FS}^{2s_0-2k+2})$ such that $\mu(0)=0$, and
\item $\vfd = \mathcal{P}_1\vfd \in H_{FS}^{2s_0+2}$ (viewed as a constant function in $t$)
\end{itemize}
to pairs $(a,b)$ with 
\begin{itemize}
\item $a=\mathcal{P}_1a \in \bigcap_{k=0}^{s_0} H^k([0,T_0];H_{FS}^{2s_0-2k}))$, and 
\item $b=\mathcal{P}_2 b \in \bigcap_{k=0}^{s_0} H^k([0,T_0];H_{FS}^{2s_0-2k})$
\end{itemize}
given by
\beq
\mathcal{F}(\tilde{f},\mu,\vfd) = \left( (\nabla_1)^2 \tilde{f} + \mathcal{P}_1 ( L_{\varphi} \tilde{f})- \vfd - i\lambda t\nabla_0 \vfd,\, (\tfrac{d}{dt} + i\lambda\nabla_0)\mu  - \mathcal{P}_2 ( L_{\varphi} \tilde{f} ) \right).
\eeq
We let $\mathcal{B}_1$ be the set of all $(\tilde{f},\mu)$ as above, so that the domain of $\mathcal{F}$ is $\mathcal{B}_1\times \mathcal{P}_1H_{FS}^{2s_0+2}$. Let $\mathcal{B}_2$ denote codomain of $\mathcal{F}$ as specified above. Linearizing $\mathcal{F}$ around $(0,0,0)$ we obtain
\beq
D\mathcal{F}_{(0,0,0)} (\tilde{f},\mu,0) = \left( (\nabla_1)^2 \tilde{f} ,\, (\tfrac{d}{dt} + i\lambda\nabla_0)\mu  \right).
\eeq
Let $(a,b)\in \mathcal{B}_2$. By Lemma \ref{lem:Q_0}, if $T_0>0$ is sufficiently small, then $(\nabla_1)^2 \tilde{f} =a$ has a unique solution $\tilde{f}\in \bigcap_{k=0}^{s_0} H^k([0,T_0];H_{FS}^{2s_0-2k+2}))$ orthogonal to the kernel of $(\nabla_1)^2$. Also, by Lemma \ref{lem:Reeb-heat-equation} there exist a unique $\mu$ with $\mu=\mathcal{P}_2\mu \in \bigcap_{k=0}^{s_0+1} H^k([0,T_0];H_{FS}^{2s_0-2k+2})$ and $\mu(0)=0$ such that $(\tfrac{d}{dt} + i\lambda\nabla_0)\mu = b$. Hence the map $D\mathcal{F}_{(0,0,0)}$ is a bijection from $\mathcal{B}_1\times\{0\}$ to $\mathcal{B}_2$. Thus by the Banach space implicit function theorem there is a neighborhood $\mathcal{U}\times \mathcal{V}$ of the origin in $\mathcal{B}_1\times \mathcal{P}_1H_{FS}^{2s_0+2}$ and a map $\mathcal{S}:\mathcal{V}\to \mathcal{B}_1$ such that $\mathcal{F}(\mathcal{S}(\vfd),\vfd)=0$ for all $\vfd \in \mathcal{V}$.

Hence, given $\vfd \in \mathcal{V}$ there exists  $\tilde{f}\in \bigcap_{k=0}^{s_0} H^k([0,T_0];H_{FS}^{2s_0-2k+2}))$ orthogonal to the kernel of $(\nabla_1)^2$ and $\mu=\mathcal{P}_2\mu \in \bigcap_{k=0}^{s_0+1} H^k([0,T_0];H_{FS}^{2s_0-2k+2})$ with $\mu(0)=0$ solving \eqref{eqn-sub-elliptic-parabolic-system}. In order to show that if $\vfd \in \mathcal{V}$ is $C^{\infty}$ and sufficiently small then $\tilde{f}$ and $\mu$ are also $C^{\infty}$ in space and time (after possibly shrinking $T_0>0$) we are going to make a similar argument again, but using the more complicated Nash-Moser framework of \cite{Hamilton1982BAMS}. We retained the above finite regularity result since its proof is much simpler; nevertheless, this result is subsumed by what follows.

Let $\mathcal{A}_1$ denote the space of those $(\tilde{f},\mu)$ that are $C^{\infty}$ in space and time, with $\tilde{f}$ orthogonal to the kernel of $(\nabla_1)^2$ and $\mu=\mathcal{P}_2\mu$, viewed as a tame Fr\'echet space with respect to the $\bigcap_{k=0}^{s} H^k([0,T_0];H_{FS}^{2s-2k})$-norms on each of the two factors. Similarly, let $\mathcal{A}_2$ denote the space of those $(a,b)$ that are $C^{\infty}$ in space and time, with $a=\mathcal{P}_1a$ and $b=\mathcal{P}_2b$, viewed as a tame Fr\'echet space in analogous fashion. Endow $\mathcal{P}_1C^{\infty}(S^3,\mathbb{C})$ with the tame Fr\'echet structure induced by the scale of Folland-Stein norms $H^{2s}_{FS}$. Then $\mathcal{F}: \mathcal{A}_1\times \mathcal{P}_1C^{\infty}(S^3,\mathbb{C}) \to \mathcal{A}_2$ is clearly a tame map. Let $(\tilde{f}_0,\mu_0,\vfd_0)\in \mathcal{A}_1\times \mathcal{P}_1C^{\infty}(S^3,\mathbb{C})$ and consider the (partial) derivative $D\mathcal{F}_{(\tilde{f}_0,\mu_0,\vfd_0)} (\tilde{f},\mu,0)$. We observe that $D\mathcal{F}_{(\tilde{f}_0,\mu_0,\vfd_0)} (\tilde{f},\mu,0) = (a,b)$ can be written as
\begin{align}
\label{eqn:system-1}
(\nabla_1)^2 \tilde{f} + \mathcal{P}_1 ( L_{\vf_0} \tilde{f} + N_{\tilde{f}_0,\mu_0,\vfd_0}\mu) &= a, \\
\label{eqn:system-2}
(\tfrac{d}{dt} + i\lambda\nabla_0)\mu - \mathcal{P}_2 ( L_{\vf_0} \tilde{f} + N_{\tilde{f}_0,\mu_0,\vfd_0}\mu) &= b,
\end{align}
where $\vf_0 = t\vfd_0 + \mu_0$ and $N_{\tilde{f}_0,\mu_0,\vfd_0}\mu = \left.\frac{d}{ds}\right|_{s=0} L_{\vf_0 + s\mu}\tilde{f}_0$. By an implicit function theorem argument similar to the above, there exists a neighborhood $\mathcal{W}$ of the origin in $\mathcal{B}_1\times \mathcal{P}_1H_{FS}^{2s_0+2}$ and a map $\mathcal{T}:\mathcal{W} \times \mathcal{B}_2 \to \mathcal{B}_1$ such that $(\tilde{f},\mu)= \mathcal{T}(\tilde{f}_0,\mu_0,\vfd_0,a,b)$ solves \eqref{eqn:system-1}--\eqref{eqn:system-2}. Let $\mathcal{W}'= \mathcal{W}\cap (\mathcal{A}_1 \times \mathcal{P}_1C^{\infty}(S^3,\mathbb{C}))$. We need to show that the restriction of $\tilde{\mathcal{S}}$ to $\mathcal{W}'\times \mathcal{A}_2$ is a smooth tame map with values in $\mathcal{A}_1$. To do this we first use \eqref{eqn:system-1} to write $\tilde{f}$ as a function of $\mu$ which then allows to view \eqref{eqn:system-2} as an equation for $\mu$ alone. Proposition \ref{lem:perturbed-Reeb-heat-equation} then gives us smoothness and tame estimates for $\mu$. Going back to the expression for $\tilde{f}$ as a function of $\mu$ we then easily obtain smoothness and tame estimates for $\tilde{f}$. In detail, from \eqref{eqn:system-1} it follows that
\beq 
(I-A)\tilde{f} = (\overline{\mathcal{Q}_0}Z_{\oneb})^2\left(a-\mathcal{P}_1N_{\tilde{f}_0,\mu_0,\vfd_0}\mu\right),
\eeq
where $A = -  (\overline{\mathcal{Q}_0}Z_{\oneb})^2\mathcal{P}_1L_{\varphi_0}$. From \eqref{eq:L} and Lemma \eqref{lem:Q_0} it's easy to see that $I-A$ is invertible on $L^2$ provided $\vf_0$ is taken to be sufficiently small in the $C^1$-norm. Moreover, it follows from \cite[Theorem 1.2.2]{Ponge2008} that the Heisenberg principal symbol of $I-A$ is invertible provided $\vf_0$ is taken to be sufficiently small in the $L^{\infty}$-norm (since then the corresponding model operator on each tangent space will be invertible in $L^2$) and hence by \cite[Propositions 5.4.2 and 5.5.9]{Ponge2008} one has a parametrix and corresponding (sub)elliptic estimates 
\beq
\norm u \norm_s \leq C_s\left(\norm(I-A)u \norm_s +\norm u\norm_0\right).
\eeq
It follows that $(I-A)^{-1}$ is a $0$th order Heisenberg pseudodifferential operator and hence (again by \cite{Ponge2008}) we have the stronger estimates
\beq
\norm (I-A)^{-1}u \norm_s \leq C_s\norm u \norm_s
\eeq
for all $s$. Hence $\tilde{a}=(I-A)^{-1}(\overline{\mathcal{Q}_0}Z_{\oneb})^2a$ is a smooth function and 
\beq \label{eq:f-tilde}
\tilde{f} = \tilde{a} - (I-A)^{-1}(\overline{\mathcal{Q}_0}Z_{\oneb})^2\mathcal{P}_1N_{\tilde{f}_0,\mu_0,\vfd_0}\mu.
\eeq
Since $(I-A)^{-1}(\overline{\mathcal{Q}_0}Z_{\oneb})^2\mathcal{P}_1N_{\tilde{f}_0,\mu_0,\vfd_0}$ is a $0$th order Heisenberg pseudodifferential operator it follows that we have the \emph{a priori} estimate
\beq\label{eq:f-tilde-norm}
\norm \tilde{f}\norm_s \leq C_s\left(\norm a\norm_s + \norm \mu \norm_s \right),
\eeq
where the constant $C_s$ depends on the $H_{FS}^{s+5}$-norms of $\tilde{f}_0$, $\mu_0$ and $\vfd_0$. Equation \eqref{eq:f-tilde} allows us to write \eqref{eqn:system-2} in the form
\beq
(\tfrac{d}{dt} + i\lambda\nabla_0)\mu + P_{\tilde{f}_0,\mu_0,\vfd_0}\mu = \tilde{b}, 
\eeq
where $\tilde{b}$ is $C^{\infty}$ and $P_{\tilde{f}_0,\mu_0,\vfd_0}$ is a second order Heisenberg pseudodifferential operator with $P_{0,0,0}=0$. By taking $\tilde{f}_0$, $\mu_0$ and $\vfd_0$ to be sufficiently small in the $C^2$-norm we can ensure that $P_{\tilde{f}_0,\mu_0,\vfd_0}$ satisfies \eqref{eq:small-perturbation} and hence by Lemma \ref{lem:perturbed-Reeb-heat-equation} we find that $\mu$ is $C^{\infty}$ and satisfies the tame estimates
\beq \label{eq:mu-tame-estimate}
\sum_{k=0}^{s+1}\norm \mu \norm_{H^k([0,T_0];H_{FS}^{2s+2-2k})} \leq C_s \left( \sum_{k=0}^{s}\norm a \norm_{H^k([0,T_0];H_{FS}^{2s-2k})} + \sum_{k=0}^{s}\norm b \norm_{H^k([0,T_0];H_{FS}^{2s-2k})} \right)
\eeq
where the constant $C_s$ depends on the $H_{FS}^{s+5}$-norms of $\tilde{f}_0$, $\mu_0$ and $\vfd_0$. Smoothness of $\tilde{f}$ then follows from \eqref{eq:f-tilde} and a similar tame estimate then follows by integrating \eqref{eq:f-tilde-norm} over $[0,T_0]$ and using \eqref{eq:mu-tame-estimate}.

It follows that the solution operator $\mathcal{T}$ for the linearized equations \eqref{eqn:system-1}--\eqref{eqn:system-2} restricts to $\mathcal{W}'\times \mathcal{A}_2$ to give a smooth tame map with values in $\mathcal{A}_1$ (after shrinking $\mathcal{W}'$ if necessary by requiring that the $C^2$-norms of $\tilde{f}_0$, $\mu_0$ and $\vfd_0$ be sufficiently small). It follows from the Nash-Moser implicit function theorem \cite[Theorem 3.3.1]{Hamilton1982BAMS} that there exists a neighborhood $\mathcal{U}_{\mathcal{A}}\times \mathcal{V}_{\mathcal{A}}$ of the origin in $\mathcal{A}_1\times \mathcal{P}_1C^{\infty}(S^3,\mathbb{C})$ and a map $\mathcal{S}_{\mathcal{A}}:\mathcal{V}_{\mathcal{A}}\to \mathcal{A}_1$ such that $\mathcal{F}(\mathcal{S}_{\mathcal{A}}(\vfd),\vfd)=0$ for all $\vfd \in \mathcal{V}_{\mathcal{A}}$. (Clearly $\mathcal{S}_{\mathcal{A}}$ is the restriction of the map $\mathcal{S}$ from the finite regularity case.) This proves the result.
\end{proof}

We conclude this section by observing that Theorem \ref{thm:vft} now follows from Theorems \ref{thm:smooth-soln} and \ref{P:embed} (cf.\ also Remark \ref{rem:analyticity-of-psi-t} for Theorem \ref{thm:vft} (ii)).

\section{Proof of Theorem \ref{thm:CL-slice-mod} and Theorem \ref{thm:slice}} \label{sec:proofs-slice-thms}

In this section we shall prove the slice theorems, Theorem \ref{thm:CL-slice-mod} and  Theorem \ref{thm:slice}, described in the introduction. The proof of  Theorem \ref{thm:CL-slice-mod} is an application of the Nash-Moser inverse function theorem (along the lines of \cite{ChengLee1995}, Theorem B). In the following the graded Fr\'echet spaces are all defined with respect to the scale of standard $L^2$-based Sobolev spaces, as in \cite{ChengLee1995}.
\begin{proof}[Proof of Theorem \ref{thm:CL-slice-mod}]
In a slight abuse of notation we identify
$\mathcal{C}\times (\mathfrak{D}'_{BE} \oplus \mathfrak{D}_0^{\perp})\times\{y_0\}$ with $\mathcal{C}\times (\mathfrak{D}'_{BE} \oplus \mathfrak{D}_0^{\perp})$.
One can define a natural action of $\mathcal{C}$ on $\mathcal{C}\times (\mathfrak{D}'_{BE} \oplus \mathfrak{D}_0^{\perp})$ so that the map $P$ in the statement of the theorem is equivariant (cf. \cite{ChengLee1995}, pp.\ 1284-1285). In order to check the conditions of the Nash-Moser inverse function theorem we need to consider the linearization of the map $P$ in Theorem \ref{thm:CL-slice-mod} at all points in a neighborhood of $(\mathrm{id},0)$ in $\mathcal{C}\times (\mathfrak{D}'_{BE} \oplus \mathfrak{D}_0^{\perp})$; by the $\mathcal{C}$-equivariance of $P$ it will suffice to consider only points of the form $(\mathrm{id},\varphi)$, as in the proof of Theorems A and B in \cite{ChengLee1995}. Recall that $\mathfrak{D}^m \cong \mathfrak{D} \times Y$, where $Y$ is the CR Cartan bundle of $S^3$ (which may be identified with $\mathrm{SU}(2,1)$ modulo its finite center). As in Cheng-Lee \cite{ChengLee1995} we write $P=(P_1,P_2)$ where $P_1$ takes values in $\mathfrak{D}$ and $P_2$ takes values in $Y$. In order to compute the linearization of $P=(P_1,P_2)$ we will make use of the local smooth tame parametrization $\Psi_e : C^{\infty}(S^3,\mathbb{R})\to \mathcal{C}$ of the contact diffeomorphism group in a neighborhood of the identity given in Theorem C of \cite{ChengLee1995} (we identify these two spaces in the calculation below, and refer to points in $C^{\infty}(S^3,\mathbb{R})$ rather than $\mathcal{C}$; so, e.g., we write $0$ instead of $\mathrm{id}$).
Using this parametrization, the linearization of $P_1$ at $(0,\varphi)$ is given by
\begin{equation}\label{eq:lin-P_1}
DP_1 (0,\varphi) (\dot{f},\dot{\varphi}) = ((\nabla_1)^2+L_{\vf})\dot{f} + \dot{\varphi}
\end{equation}
for $\dot{f}\in C^{\infty}(S^3,\mathbb{R})$, $\dot{\vf}\in \mathfrak{D}'_{BE} \oplus \mathfrak{D}_0^{\perp}$ where $L_{\vf}$ is as in \eqref{eq:L} (cf. \cite{ChengLee1995}, equation (5.1); note that we are calculating exclusively in terms of the frame $Z_1$). We decompose $C^{\infty}(S^3,\mathbb{R})$ as a direct sum $C_{\mathfrak{CR}}^{\infty}(S^3,\mathbb{R})\oplus C_{\perp}^{\infty}(S^3,\mathbb{R})$ where $C_{\mathfrak{CR}}^{\infty}(S^3,\mathbb{R}) = \ker (\nabla_1)^2 \cap C^{\infty}(S^3,\mathbb{R}) = \bigoplus_{p,q\in \{0,1\}}H_{p,q}\cap C^{\infty}(S^3,\mathbb{R})$ is the $8$-dimensional space of potential functions for the infinitesimal CR automorphisms of the standard $S^3$, and $C_{\perp}^{\infty}(S^3,\mathbb{R}) = \overline{\bigoplus_{p,q\neq 0,1}H_{p,q}}\cap C^{\infty}(S^3,\mathbb{R})$ is the $L^2$ orthogonal complement of $C_{\mathfrak{CR}}^{\infty}(S^3,\mathbb{R})$ in $C^{\infty}(S^3,\mathbb{R})$. We also decompose (cf. \cite[p.\ 833]{BurnsEpstein1990b})
\beq\mathfrak{D} = \mathfrak{D}'_{BE}\oplus (Z_1)^2(C^{\infty}(S^3,\mathbb{R})) \oplus \mathfrak{D}_0^{\perp}
\eeq
and let $\Pi: \mathfrak{D} \to  (Z_1)^2(C^{\infty}(S^3,\mathbb{R}))=(Z_1)^2(C_{\perp}^{\infty}(S^3,\mathbb{R}))$ denote the corresponding projection (the projection is oblique, but is bounded in $H_{FS}^s$ for every $s$, cf.\ \cite[page 833]{BurnsEpstein1990b}). Note that if $\dot{\vf}\in \mathfrak{D}'_{BE} \oplus \mathfrak{D}_0^{\perp}$ then $\Pi \dot{\vf} =0$. We construct a family of inverse maps $VP (0,\varphi)$ to the family of linearized maps $DP (0,\varphi)$ as follows. Given $(K,X)\in \mathfrak{D}\times \mathfrak{su}(2,1) \cong T_{(0,\vf)} \mathfrak{D}^m$ we need to solve uniquely the following linear equations
\begin{align}
\label{eq:PiDP1} \Pi DP_1(0,\vf) (\dot{f},\dot{\vf}) &= (\nabla_1)^2 \dot{g} + \Pi L_{\vf} \dot{g} + \Pi L_{\vf} \dot{h} = \Pi K\\
\label{eq:id-PiDP1}(\mathrm{id}-\Pi) DP_1(0,\vf) (\dot{f},\dot{\vf}) & = \dot{\vf} + (\mathrm{id}-\Pi)L_{\vf} \dot{f} = (\mathrm{id}-\Pi)K\\
\label{eq:DP2}DP_2 (0,\vf) (\dot{f},\dot{\vf}) &= X
\end{align}
where $\dot{f} = \dot{g}+\dot{h}$ with $\dot{h}\in C_{CR}^{\infty}(S^3,\mathbb{R})$ and $\dot{g} \in C_{\perp}^{\infty}(S^3,\mathbb{R})$. As in the proof of Proposition \ref{prop:convergence-ft}, by an elliptic regularity argument the map
\begin{equation}
(\nabla_1)^2 + \Pi L_{\vf} : C_{\perp}^{\infty}(S^3,\mathbb{R}) \to (Z_1)^2(C_{\perp}^{\infty}(S^3,\mathbb{R}))
\end{equation}
has a smooth tame solution operator $((\nabla_1)^2 + \Pi L_{\vf})^{-1}$ for $\vf$ sufficiently small (in the $C^1$ sense) with smooth tame dependence on $\vf$. Using this solution operator we may solve \eqref{eq:PiDP1} for $\dot{g}$, viewing $\dot{h} \in C^{\infty}_{\mathfrak{CR}} (S^3, \mathbb{R})$ as a free $8$-dimensional parameter for now. One may then simply choose $\dot{\vf}$ to satisfy \eqref{eq:id-PiDP1}, again viewing $\dot{h}$ as a parameter. Plugging the solutions for $\dot{g}$ and $\dot{\vf}$ into \eqref{eq:DP2} yields a finite dimensional equation to be solved for $\dot{h}$ in terms of $X$; solvability for small $\vf$ follows easily by the standard finite dimensional inverse function theorem after checking that this map is injective at $(0,\vf)=(0,0)$, where the map becomes an identification between potentials for infinitesimal CR automorphisms and the corresponding elements of $\mathfrak{su}(2,1)$. This establishes the existence of a smooth tame family $VP(0,\vf)$ of inverses to the family $DP(0,\vf)$ of linearized maps, for sufficiently small $\vf$. Part (i) of the theorem now follows by the Nash-Moser inverse function theorem.

Part (ii) follows easily from inspecting the linearized action of the contact diffeomorphisms on the slice.
\end{proof}

We claim that the restriction of the map $P:\mathcal{C}\times (\mathfrak{D}'_{BE} \oplus \mathfrak{D}_0^{\perp})\times \{y_0\} \to \mathfrak{D}^m$ from Theorem \ref{thm:CL-slice-mod} to $P_{emb}:\mathcal{C}\times \mathfrak{D}'_{BE} \times \{y_0\} \to \mathfrak{D}^m$ locally parametrizes the set of marked embeddable deformations of the standard CR sphere. This will be proved in Proposition \ref{prop:emb-slice} below. An argument from \cite{Bland1994} (which will be fleshed out in the proof of Proposition \ref{prop:emb-slice} below) shows that the natural map $\mathcal{C}\times \mathfrak{D}_{BE} \times \{y_0\} \to \mathfrak{D}^m_{emb}$ is surjective; but this map only becomes injective after we further restrict to the map $P_{emb}:\mathcal{C}\times \mathfrak{D}'_{BE} \times \{y_0\} \to \mathfrak{D}^m_{emb}$. In order to show that the restricted map $P_{emb}$ is surjective we need the following lemma. Let $\mathfrak{D}_{\mathrm{cd}}$ denote the set of smooth deformation tensors on the standard CR $3$-sphere whose spherical harmonic decomposition is supported on the critical diagonal, i.e. the deformation tensors $\vf = \sum_{p} \vf_{p,p+4}$. Note that $\mathfrak{D}_{\mathrm{cd}}$ is precisely the space of deformation tensors corresponding to $S^1$-invariant CR structures. Let $\mathfrak{D}'_{\mathrm{cd}} =\{\, \vf\in \mathfrak{D}_{\mathrm{cd}} \,|\, \im  (\nabla^1)^2 \vf =0\, \}$.
\begin{lemma}\label{lem:CL-slice-mod-S1}
Let $\vf_0 \in \mathfrak{D}_{\mathrm{cd}}$ be sufficiently small. Then there exists an $S^1$-equivariant contact diffeomorphism of $S^3$ (unique modulo $S^1$-equivariant automorphisms of the CR sphere) pulling the CR structure  corresponding to $\vf_0$ back to one with deformation tensor $\tilde{\vf}_0 \in \mathfrak{D}'_{\mathrm{cd}}$. Moreover, the contact diffeomorphism can be chosen to smoothly depend on $\vf_0$.
\end{lemma}
\begin{proof}
The argument is to establish a slice theorem essentially as in the proof of Theorem \ref{thm:CL-slice-mod}, but can be made slightly simpler due to the relevant automorphism group now being compact so that markings are not needed (this is analogous to the situation of Theorem A in \cite{ChengLee1995}). Let $\mathcal{C}^{S^1}$ denote the space of $S^1$-equivariant contact diffeomorphisms of $S^3$ and let $H\subseteq\mathrm{Aut}(S^3)$ denote the subgroup of group of $S^1$-equivariant automorphism of the CR sphere $S^3$. By restricting the local parametrization $\Psi_e : C^{\infty}(S^3,\mathbb{R})\to \mathcal{C}$ of the contact diffeomorphism group (in a neighborhood of the identity) given in Theorem C of \cite{ChengLee1995} to $S^1$-invariant functions we obtain a smooth tame parametrization $C^{\infty}(S^3,\mathbb{R})^{S^1} \to \mathcal{C}^{S^1}$ of $\mathcal{C}^{S^1}$ in a neighborhood of the identity. As in Theorem D of \cite{ChengLee1995}, by restricting this map to the space $\mathfrak{W}$ of functions $f$ in $C^{\infty}(S^3,\mathbb{R})$ with spherical harmonic decomposition of the form $f= \sum_{p\geq 2} f_{p,p}$ (i.e. functions $f\in C^{\infty}(S^3,\mathbb{R})^{S^1}$ with $f_{0,0}=f_{1,1}=0$) we obtain, as the image of the restricted map, a local slice $\mathcal{W} \subseteq \mathcal{C}^{S^1}$ for the coset space $\mathcal{C}^{S^1}/H$.

Let $P_0 : \mathcal{W}\times \mathfrak{D}'_{\mathrm{cd}} \to \mathfrak{D}_{\mathrm{cd}}$ denote the natural map (where the contact diffeomorphism acts by pullback on the CR structure corresponding to the deformation tensor). One can define a natural action of $\mathcal{C}^{S^1}$ on $\mathcal{W}\times \mathfrak{D}'_{\mathrm{cd}}$ so that the map $P_0$ is equivariant (cf. \cite{ChengLee1995}, pp.\ 1284-1285). In order to check the conditions of the Nash-Moser inverse function theorem we need to consider the linearization of $P_0$ at all points in a neighborhood of $(\mathrm{id},0)$ in $\mathcal{C}^{S^1}\times \mathfrak{D}_{\mathrm{cd}}$; by the equivariance of $P_0$ it will suffice to consider only points of the form $(\mathrm{id},\varphi)$, as in the proof of Theorem A in \cite{ChengLee1995}. In order to compute the linearization we will make use of the local smooth tame parametrization of $\mathcal{W}$ by the space of functions $\mathfrak{W}$ (we identify these two spaces in the calculation below, and refer to points in $\mathfrak{W}$ rather than $\mathcal{W}$).
Using this parametrization, the linearization of $P_0$ at $(0,\varphi)$ is given by
\begin{equation}\label{eq:lin-P_0}
DP_0 (0,\varphi) (\dot{f},\dot{\varphi}) = ((\nabla_1)^2+L_{\vf})\dot{f} + \dot{\varphi}
\end{equation}
for $\dot{f}\in \mathfrak{W}$, $\dot{\vf}\in \mathfrak{D}'_{cd}$ where $L_{\vf}$ is as in \eqref{eq:L} (cf. \cite{ChengLee1995}, equation (5.1)). We construct a family of inverse maps $VP_0 (0,\varphi)$ to the family of linearized maps $DP_0 (0,\varphi)$ as follows. Let $\Pi_0 : \mathfrak{D}_{cd} \to \mathfrak{D}'_{cd}{}^{\perp}  \subseteq \mathfrak{D}_{cd}$ denote the $L^2$ orthogonal projection. (Note that $\mathfrak{D}'_{cd}{}^{\perp}$ is the image of $\mathfrak{W}$, or equivalently of $C^{\infty}(S^3,\mathbb{R})^{S^1}$, under $(Z_1)^2$.) For $\chi\in \mathfrak{D}_{cd}$ we write $\chi = \chi_1+\chi_2$ where $\chi_1 = \Pi_0 \chi$ and $\chi_2 = (\mathrm{id}-\Pi_0)\chi \in \mathfrak{D}'_{cd}$. Given $\chi\in \mathfrak{D}_{cd}$ we first solve
$$
((\nabla_1)^2 + \Pi_0 L_{\vf})\dot{f} = \chi_1
$$
using the same argument as in the proof of Proposition \ref{prop:convergence-ft}. As in the proof of Proposition \ref{prop:convergence-ft}, by an elliptic regularity argument the map
$((\nabla_1)^2+ \Pi_0 L_{\vf}) : \mathfrak{W}\to \mathfrak{D}'_{cd}{}^{\perp}$ has a smooth tame inverse for sufficiently small $\vf$. Since we are free to choose $\dot{\varphi} \in \mathfrak{D}'_{cd}$ to solve the $(\mathrm{id}-\Pi_0)$ projection of
$$
((\nabla_1)^2+L_{\vf})\dot{f} + \dot{\varphi} = \chi
$$
we obtain a smooth family of inverses $VP_0 (0,\varphi)$.
Thus, by the Nash-Moser inverse function theorem, given any sufficiently small deformation tensor $\vf_0$ there exists a (unique small) $S^1$-equivariant contact diffeomorphism (in $\mathcal{W}$) pulling the corresponding CR structure back to one with deformation tensor $\tilde{\vf}_0 \in \mathfrak{D}'_{cd}$.

\end{proof}

\begin{proposition}\label{prop:emb-slice}
Fix any marking $y_0$ of the standard CR sphere. The natural map $P_{emb}:\mathcal{C}\times \mathfrak{D}'_{BE}\times \{y_0\} \to \mathfrak{D}^m$ is a local bijection from an open neighborhood of $(\mathrm{id},0,y_0)$ to an open neighborhood of $(0,y_0)$ in the subset $\mathfrak{D}^m_{emb}$ of marked embeddable deformations of the standard CR $3$-sphere.
\end{proposition}
\begin{proof}
That $P_{emb}$ maps $\mathcal{C}\times \mathfrak{D}'_{BE}\times\{y_0\}$ into $\mathfrak{D}^m_{emb}$ follows from \cite[Theorem 5.3]{BurnsEpstein1990b}, cf. also Theorem \ref{thm:BE-parametric} in this paper. Injectivity then follows from Theorem \ref{thm:CL-slice-mod} above. To see that the map is surjective, we first let $\vf$ be an embeddable deformation, with $\vf$ sufficiently small such that there is an embedding $\Phi:S^3\to \mathbb{C}^2$ with image a strictly convex hypersurface near the standard sphere realizing $\vf$ (i.e. such that $\Phi_*(Z_1+\vf_1{}^{\oneb}Z_{\oneb})$ is a $(1,0)$-vector field along the image of $\Phi$). The hypersurface $M=\Phi(S^3)$ bounds a convex domain $\Omega$ which has Kobayashi indicatrix $B\subseteq T_0 \mathbb{C}^2 \cong \mathbb{C}^2$ based at $0$. Let $\Psi:B\to \Omega$ denote the circular representation of $\Omega$ \cite{Lempert1981, BlandDuchamp1991}, which is smooth up to the boundary (and away from the origin). By \cite{BlandDuchamp1991}, equation (3.5), $\Psi|_{\partial B}$ is a contact diffeomorphism from $\partial B$ to $M =\partial\Omega$  and so the CR structure on $M$ pulls back to a deformation $\vf_{M,\partial B}$ of the CR structure on $\partial B$; moreover, $\overline{\vf_{M,\partial B}}$ (when expressed in terms of an $S^1$-invariant framing) has only positive Fourier coefficients with respect to the natural $S^1$ action on $\partial B$ \cite{BlandDuchamp1991}.

The radial projection from $\partial B$ to $S^3$ is clearly $S^1$-equivariant, but is not (in general) a contact diffeomorphism (so it can be thought of as endowing $S^3$ with a second $S^1$-invariant contact distribution). We may correct for this possible discrepancy by using the $S^1$-invariant version of Gray's classical theorem (since our two contact distributions on $S^3$ are isotopic through $S^1$-invariant contact distributions), which tells us that there exists an $S^1$-invariant contact diffeomorphism from $\partial B$ to $S^3$. This contact diffeomorphism allows us to push forward the intrinsic CR structure on $\partial B$ to an $S^1$-invariant CR structure on $S^3$ compatible with the standard contact distribution $H$; with respect to the standard frame $Z_1$ on $S^3$ this CR structure has deformation tensor $\tilde{\vf}_{\partial B, S^3} \in \mathfrak{D}_{cd}$ (by $S^1$-invariance). Using Lemma \ref{lem:CL-slice-mod-S1} there exists an $S^1$-equivariant contact diffeomorphism which pulls the CR structure corresponding to $\tilde{\vf}_{\partial B, S^3}$ back to one with deformation tensor $\vf_{\partial B, S^3} \in \mathfrak{D}'_{cd}$. We shall denote by $\psi$ the contact diffeomorphism $\partial B \to S^3$ that pushes forward the CR structure on $\partial B$ to the one on $S^3$ with deformation tensor $\vf_{\partial B, S^3} \in \mathfrak{D}'_{cd}$. Using this contact diffeomorphism we may also push forward the CR structure with deformation tensor $\vf_{M,\partial B}$ from $\partial B$ to one on $S^3$ with deformation  tensor $\vf_{M,S^3}$. Note that $\vf_{M,S^3}$ (which describes the CR structure of $M$ relative to $S^3$) differs from $\psi_* \vf_{M,\partial B}$, since the latter is the deformation tensor for the CR structure of $M$ relative to the CR structure of $\partial B$ after we have identified $\partial B$ and $M$ with $S^3$ using the circular representation and the map $\psi$. Knowing that the deformation tensor of $\partial B$ relative to $S^3$ is $\vf_{\partial B, S^3}$ and the deformation tensor of $M$ relative to $\partial B$ is $\psi_* \vf_{M,\partial B}$, it is easy to show that
\beq \label{eq:chain-of-defs}
\vf_{M,S^3} = \frac{\psi_*\vf_{M,\partial B}+\vf_{\partial B, S^3}}{1+(\psi_*\vf_{M,\partial B})\cdot\overline{\vf_{\partial B, S^3}}}.
\eeq
We now claim that $\vf_{M,S^3} \in \mathfrak{D}'_{BE}$. Choose a unitary $S^1$-invariant framing $Z_1^{\partial B}$ for the CR structure on $\partial B$ and similarly a unitary $S^1$-invariant framing $Z_1^0$ on the standard CR sphere $S^3$.
Working in these frames the identity \eqref{eq:chain-of-defs} becomes an identity of functions, and $\psi_*\vf_{M,\partial B}$ is just $\vf_{M,\partial B} \circ \psi^{-1}$.
Since $\psi$ is $S^1$-equivariant, it follows from \eqref{eq:chain-of-defs} that $\vf_{M,S^3}$ has only non-positive Fourier coefficients, and moreover, the zeroth Fourier component of $\vf_{M,S^3}$ is simply $\vf_{\partial B, S^3}$.
 When expressed in the standard framing $Z_1, Z_{\oneb}$ of $S^3$ (which are not $S^1$ invariant since $\mathcal{L}_{T} Z_1 = -2i Z_1$) then Fourier coefficients are shifted by $-4$ $(= -2 -2)$ and hence, viewed in this frame, the deformation $\vf_{M,S^3}$ lies in $\mathfrak{D}_{BE}$.
Moreover, since the spherical harmonic coefficients of $\vf_{M,S^3}$ agree with $\vf_{\partial B, S^3}$ along the critical diagonal (which corresponded to the zeroth Fourier mode when using $S^1$-invariant framings) and $\vf_{\partial B, S^3}\in \mathfrak{D}'_{cd}$ we have that $\vf_{M, S^3} \in \mathfrak{D}'_{BE}$ when expressed in terms of the standard framing for $S^3$. This establishes that $(\vf, y_1)$ is in the image of $P_{emb}$ for some marking $y_1$.

It remains to show that $(\vf, y)$ is in the image of $P_{emb}$ for all markings $y$ in a uniform neighborhood of $y_0$ (for $\vf$ sufficiently small). Note that in the preceding argument we could have chosen a different base point $p$ for the Kobayashi indicatrix. Note also that $\mathrm{Aut}(S^3) = \mathrm{Aut}(\mathbb{B}^2)$ acts simply and transitively on the set of pointed frames in $\mathbb{B}^2$. Using this we can act on the marking $y_1$ of $(\vf, y_1)$ while keeping $\vf$ fixed as follows. Given a point $p\in \mathbb{B}^2$ and a unitary frame $(e_1,e_2)$ for $T_p\mathbb{C}^2$ we repeat the above construction of $\vf_{M, S^3} \in \mathfrak{D}'_{BE}$ but now use the Kobayashi indicatrix $B_p$ centered at $p$ and identify $\partial B_p \subseteq T_p \mathbb{C}^2$ with $\partial B = \partial B_0 \subseteq T_0 \mathbb{C}^2$ using the linearization of the automorphism of $\mathbb{B}^2$ that takes $(p,(e_1,e_2))$ to the point $0$ with the standard frame. In this way we obtain a family $(\psi_s,\vf_s)$ of points in $\mathcal{C}\times \mathfrak{D}'_{BE}$ parametrized by $s\in \mathrm{Aut}(S^3)$ whose images under $P_{emb}$ are all of the form $(\vf, y_s)$. Since $\mathrm{Aut}(S^3)\cong Y$ is finite dimensional and the map $s \mapsto y_s$ in the special case $\vf =0$ is just the natural identification, the map $s \mapsto y_s$ is a local diffeomorphism for sufficiently small $\vf$. This proves the result.
\end{proof}

Theorem \ref{thm:slice} now follows from Proposition \ref{prop:emb-slice} and Theorem \ref{thm:CL-slice-mod}.

\begin{remark}
For comparison, we note that in \cite[Theorem 14.2 and Theorem 15.1]{Bland1994} Bland gives a normal form for CR structures and for embeddable CR structures on $S^3$ near the standard structure with respect to the action of contact diffeomorphisms and $S^1$-equivariant diffeomorphisms (which do not preserve the contact distribution). In the notation of the proof of Proposition \ref{prop:emb-slice} Bland's normal form for the embeddable deformation $\vf$ is obtained by pushing $\vf_{M,\partial B}$ forward to $S^3$ using the radial projection from $\partial B$ to $S^3$ and viewing this as a deformation of the CR structure of $\partial B$ pushed forward to $S^3$ (recall that via this identification the contact distribution of $\partial B$ is $S^1$-invariant but does not in general match the standard contact distribution of $S^3$, which is why $S^1$-equivariant diffeomorphisms are needed for this normalization). Our approach has been to keep the underlying contact structure fixed, which allows us to view the deformation in normal form as a deformation of the standard CR structure.
\end{remark}
\vfill

\subsection{Declarations}
\begin{itemize}

\item Funding. The second author was supported in part by the NSF grants DMS-1600701 and DMS-1900955

\item Conflicts of interest/Competing interests. On behalf of all authors, the corresponding author states that there is no conflict of interest.

\item Availability of data and material (data transparency). N/A.

\item Code availability. N/A.

\end{itemize}

\bibliographystyle{plain}
\newcommand{\noopsort}[1]{}

\end{document}